\documentclass[12pt]{article}
\usepackage{amsfonts, amsmath, amssymb, latexsym, eucal, amscd} 

\usepackage[dvips]{pict2e}

\usepackage{amsthm}
\usepackage{amscd}

\usepackage[dvipdfmx]{graphics}
\usepackage{cite}

\newtheorem{definition}{Definition}[section]
\newtheorem{theorem}[definition]{Theorem}
\newtheorem{lemma}[definition]{Lemma}
\newtheorem{corollary}[definition]{Corollary}
\newtheorem{remark}[definition]{Remark}
\newtheorem{example}[definition]{Example}
\newtheorem{conjecture}[definition]{Conjecture}
\newtheorem{problem}[definition]{Problem}
\newtheorem{note}[definition]{Note}

\newtheorem{proposition}[definition]{Proposition}

\begin{document} 

\title{\bf
The $\mathbb Z_3$-Symmetric Down-Up algebra
}
\author{
Paul Terwilliger
\\
\\
{\it In memory of Georgia Benkart (1947--2022)}
}
\date{}

\maketitle
\begin{abstract} 
In 1998, Georgia Benkart and Tom Roby introduced the down-up algebra $\mathcal A$. The algebra $\mathcal A$ is associative, noncommutative, and infinite-dimensional.
 It is defined by two generators $A,B$ and two relations called the down-up relations.
In the present paper, we introduce  the $\mathbb Z_3$-symmetric down-up algebra $\mathbb A$. 
We define $\mathbb A$ by generators and relations. There are three generators $A,B,C$ and any two of these satisfy the down-up relations. We describe how $\mathbb A$ is related to some
familiar algebras in the literature, such as the Weyl algebra, the Lie algebras
 $\mathfrak{sl}_2$ and $\mathfrak{sl}_3$, the   $\mathfrak{sl}_3$ loop algebra, the Kac-Moody Lie algebra $A^{(1)}_2$, the $q$-Weyl algebra, the
 quantized enveloping algebra $U_q(\mathfrak{sl}_2)$,  and the quantized enveloping algebra $U_q (A^{(1)}_2)$. 
 We give some open problems and conjectures.
\bigskip

 \end{abstract}
 
 \section{Introduction}
 The work of Georgia Benkart contains many gems \cite{gems}. Some of these gems require considerable sophistication to appreciate, but one gem that seems to
 attract everyone is the concept of a down-up algebra. This algebra was introduced in \cite{benkart} by Benkart and Tom Roby.
 The down-up algebra, which we will denote by 
 $\mathcal A(\alpha, \beta, \gamma)$, is determined by three scalar parameters $\alpha, \beta, \gamma$ that  take arbitrary values in the ground field.
 The algebra $\mathcal  A(\alpha, \beta, \gamma)$ is associative, noncommutative, and infinite-dimensional.
 It is defined by two generators $A, B$ and two relations called the down-up relations:
 \begin{align*}
B A^2 &= \alpha ABA + \beta A^2B + \gamma A, \qquad \qquad
B^2 A = \alpha BAB + \beta AB^2 + \gamma B.
\end{align*}
\noindent As explained in \cite{benkart}, the down-up relations are motivated by a topic in algebraic combinatorics, concerning partially ordered sets.
The relations describe the raising map/lowering map interaction for a poset that is differential \cite{stanley} or uniform \cite{uniform}.
\medskip

\noindent  Down-up algebras have applications  well beyond combinatorics. As explained in \cite{benkart}, the down-up algebras are
related to many familiar algebras in the literature, such as the Weyl algebra, $q$-Weyl algebra, the quantum plane, the Lie algebra $\mathfrak{sl}_2$, the
Heisenberg Lie algebra, and the quantized enveloping algebra $U_q(\mathfrak{sl}_3)$.
\medskip

\noindent  A down-up algebra is reminiscent of the enveloping algebra of a Lie algebra.
As explained  in \cite{benkart}, a down-up algebra has a basis of  Poincar\'e--Birkhoff--Witt type (see Lemma \ref{lem:basis}  below)
and a representation theory that involves Verma modules, 
 highest weight modules, lowest weight modules, 
  and category $\mathcal O$ modules.
\medskip

\noindent The article \cite{benkart} had a large impact;  it is
presently cited 96 times according to MathSciNet. The citing papers determine many features of a down-up algebra, such as
  the automorphisms
  \cite{carv2},
  the center \cite{hildebrand, kulkarni, zhao},
the Hopf algebra structures
\cite{benkWith, kirk2},
the  primitive ideals \cite{DAJordan2,praton, praton2},
the injective modules
\cite{carv1},
the simple modules \cite{cas, carv3},
and the Whittaker modules
\cite{benkOndrus},
to name a few.
 \medskip
 
 \noindent In the present paper, we introduce a variation of the down-up algebra $\mathcal A(\alpha, \beta, \gamma)$ called the $\mathbb Z_3$-symmetric down-up algebra $\mathbb A(\alpha, \beta, \gamma)$. 
This algebra is defined by generators $A,B,C$ and the following relations:
\begin{align*}
&B A^2 = \alpha ABA + \beta A^2B + \gamma A, \qquad \qquad 
B^2 A = \alpha BAB + \beta AB^2 + \gamma B,
\\
&CB^2  = \alpha BCB + \beta B^2 C + \gamma B, \qquad \qquad 
C^2 B = \alpha CBC + \beta BC^2 + \gamma C,
\\
&AC^2  = \alpha CAC + \beta C^2 A + \gamma C, \qquad \qquad 
A^2 C = \alpha ACA + \beta CA^2 + \gamma A.
\end{align*}
\noindent As we will explain in Section 20 below, the algebra $\mathbb A(\alpha, \beta, \gamma)$ is motivated by the concept of a lowering-raising triple of linear transformations \cite{lrt}.
\medskip

\noindent In the present paper, our main goal is to describe how the algebra $\mathbb A(\alpha, \beta, \gamma)$ is related to some
familiar algebras in the literature, such as the Weyl algebra, the Lie algebras
 $\mathfrak{sl}_2$ and $\mathfrak{sl}_3$, the   $\mathfrak{sl}_3$ loop algebra, the Kac-Moody Lie algebra $A^{(1)}_2$, the $q$-Weyl algebra, the
 quantized enveloping algebra $U_q(\mathfrak{sl}_2)$,  and the quantized enveloping algebra $U_q (A^{(1)}_2)$. 
 As we pursue the main goal, we find it useful to introduce a Lie algebra $\mathbb L(\gamma)$ and an algebra $\mathbb R(\theta)$ that may be of independent interest.
 For a scalar $\gamma$ the Lie algebra $\mathbb L(\gamma)$ is defined by generators $A,B,C$ 
and the following relations:
\begin{align*}
&\lbrack A, \lbrack A, B \rbrack \rbrack = \gamma A, \qquad \qquad \lbrack B, \lbrack B, A \rbrack \rbrack = \gamma B,  \\
&\lbrack B, \lbrack B, C \rbrack \rbrack =  \gamma B, \qquad \qquad \lbrack C, \lbrack C, B \rbrack \rbrack = \gamma C,    \\
&\lbrack C, \lbrack C, A \rbrack \rbrack = \gamma C, \qquad \qquad \lbrack A, \lbrack A, C \rbrack \rbrack = \gamma A. 
\end{align*}
We call $\mathbb L(\gamma) $ the  $\mathbb Z_3$-symmetric down-up Lie algebra with parameter $\gamma$.
As we will see, the enveloping algebra of $\mathbb L(\gamma)$ is isomorphic to $\mathbb A(2, -1, \gamma)$.
For a scalar $\theta$ the algebra $\mathbb R(\theta)$ is defined by generators $A,B,C$ and the following relations:
\begin{align*}
&A^2=0, \qquad \qquad B^2 =0, \qquad \qquad C^2=0, \\
& ABA=\theta A, \qquad BCB=\theta B, \qquad CAC=\theta C,  \\
& BAB = \theta B, \qquad CBC=\theta C, \qquad ACA=\theta A. 
\end{align*}
We call
$\mathbb R(\theta)$ the  reduced $\mathbb Z_3$-symmetric down-up algebra with parameter $\theta$. We give a basis for $\mathbb R(\theta)$.
We show that $\mathbb R(-\alpha^{-1} \gamma)$ is a homomorphic image of $\mathbb A(\alpha, \beta, \gamma)$, provided that $\alpha \not=0$.
We also describe how $\mathbb R(\theta)$ is related to $\mathbb L(-2\theta)$. 
\medskip

\noindent At the  end of the paper, we give some open problems and conjectures.
\medskip

 \noindent The paper is organized as follows. Section 2 contains some preliminaries. In Section 3, we review the down-up algebra introduced by Benkart and Roby.
 In Section 4, we introduce the $\mathbb Z_3$-symmetric down-up algebra $\mathbb A=\mathbb A(\alpha, \beta, \gamma)$, and describe its basic properties.
 In Section 5, we describe some symmetries of $\mathbb A$ that involve changing the parameters.
 In Section 6, we describe a $\mathbb Z_2$-grading of $\mathbb A$. 
 In Section 7, we consider two extreme cases that illustrate how the structure of $\mathbb A(\alpha, \beta, \gamma)$ depends on $\alpha, \beta, \gamma$.
 In Section 8, we describe several algebra homomorphisms that involve $\mathbb A$.
 In Section 9, we describe how $\mathbb A$ is related to the Weyl algebra.
 In Section 10, we introduce the Lie algebra $\mathbb L(\gamma)$ and describe how it is related to $\mathbb A$.
 In Section 11, we introduce the algebra $\mathbb R(\theta)$ and describe how it is related to $\mathbb A$.
 In Sections 12--18, we describe how $\mathbb A$ is related to the  Lie algebras
 $\mathfrak{sl}_2$ and $\mathfrak{sl}_3$, the   $\mathfrak{sl}_3$ loop algebra, the Kac-Moody Lie algebra $A^{(1)}_2$, the $q$-Weyl algebra, the
 quantized enveloping algebra $U_q(\mathfrak{sl}_2)$,  and the quantized enveloping algebra $U_q (A^{(1)}_2)$.
In Section 19, we make some observations about one of the extreme cases discussed in Section 7.
In Section 20,  we explain how the algebra $\mathbb A$ is motivated by the concept of a lowering-raising triple of linear transformations.
  Section 21 contains some directions for future research.
  
\section{Preliminaries}

\noindent  We now begin our formal argument. Throughout the paper, we adopt the following assumptions and notational conventions.
Recall the natural numbers $\mathbb N=\lbrace 0,1,2,\ldots \rbrace$ and
integers $\mathbb Z = \lbrace 0, \pm 1, \pm 2, \ldots \rbrace$.
Let $\mathbb F$ denote an algebraically closed field with characteristic zero. 
An element in $\mathbb F$ is  called a scalar.
Every vector space and tensor product that we mention,  is understood to be over $\mathbb F$.
Every algebra without the Lie prefix that we mention, is understood to be associative, over $\mathbb F$, and have a multiplicative identity.
A subalgebra has the same multiplicative identity as the parent algebra. For an integer $n\geq 1$,  ${\rm Mat}_n(\mathbb F)$
denotes the algebra of  $n \times n$ matrices that have all coefficients in $\mathbb F$. 
Let $t$ denote an indeterminate. Let $\mathbb F\lbrack t, t^{-1}\rbrack$ denote the algebra of Laurent polynomials in $t$ that have all coefficients in $\mathbb F$.
 
 \section{Review of the down-up algebra}
 
 \noindent In this section, we review the down-up algebra introduced by Benkart and Roby \cite{benkart}.
 \begin{definition} \label{def:du1} \rm (See \cite[p.~308]{benkart}.) For $\alpha, \beta, \gamma \in \mathbb F$ the algebra $\mathcal A= \mathcal A(\alpha, \beta, \gamma)$ is defined by generators $A,B$ and the following relations:
\begin{align*}
B A^2 = \alpha ABA + \beta A^2B + \gamma A,\qquad \qquad B^2 A = \alpha BAB + \beta AB^2 + \gamma B.
\end{align*}
\noindent The algebra $\mathcal A$ is called the {\it  down-up algebra} with parameters $\alpha, \beta, \gamma$.
\end{definition}

\begin{lemma}  \label{lem:basis} {\rm (See \cite[Theorem~3.1]{benkart}.)} The following is a basis for the vector space $\mathcal A$:
\begin{align} \label{eq:basisL}
A^i (BA)^j B^k \qquad \qquad (i,j,k \in \mathbb N).
\end{align}
\end{lemma}

\begin{corollary} \label{cor:NC} The algebra $\mathcal A$ is infinite-dimensional and noncommutative.
\end{corollary}
\begin{proof} The basis  \eqref{eq:basisL} has infinite cardinality. This basis contains $AB$ and $BA$, so $AB\not=BA$.
\end{proof}

\begin{lemma} \label{lem:primary} {\rm (See \cite[Example~2.6]{benkart}.)} There exists an algebra homomorphism $\mathcal A \to \mathbb F$ that sends $A\mapsto 0$ and $B \mapsto 0$.
\end{lemma}


\begin{lemma} {\rm (See \cite[p.~314]{benkart}.)} For the algebra $\mathcal A$ we have $\lbrack AB, BA\rbrack=0$.
\end{lemma}

 \section{The $\mathbb Z_3$-symmetric down-up algebra $\mathbb A$}
 \noindent  In this section, we introduce the  $\mathbb Z_3$-symmetric down-up algebra $\mathbb A$.

\begin{definition} \label{def:bbA} \rm For $\alpha, \beta, \gamma \in \mathbb F$ the algebra $\mathbb A= \mathbb A(\alpha, \beta, \gamma)$ is defined by generators $A,B,C$ and the following relations:
\begin{align}
&B A^2 = \alpha ABA + \beta A^2B + \gamma A, \qquad \qquad \label{eq:z3du1}
B^2 A = \alpha BAB + \beta AB^2 + \gamma B,
\\
&CB^2  = \alpha BCB + \beta B^2 C + \gamma B, \qquad \qquad  \label{eq:z3du2}
C^2 B = \alpha CBC + \beta BC^2 + \gamma C,
\\
&AC^2  = \alpha CAC + \beta C^2 A + \gamma C, \qquad \qquad 
A^2 C = \alpha ACA + \beta CA^2 + \gamma A.
\label{eq:z3du3}
\end{align}
\noindent We call $\mathbb A$ the {\it $\mathbb Z_3$-symmetric down-up algebra} with parameters $\alpha, \beta, \gamma$. We call $A,B,C$ the {\it standard generators} of $\mathbb A$.
We call the relations  \eqref{eq:z3du1}--\eqref{eq:z3du3}  the {\it $\mathbb Z_3$-symmetric down-up relations} with parameters $\alpha$, $\beta$, $\gamma$.
\end{definition}

\noindent We mention some basic facts about $\mathbb A$. For the next seven lemmas, the proof is routine and omitted.

\begin{lemma} \label{lem:natH} For $\alpha, \beta, \gamma \in \mathbb F$ there exists an algebra homomorphism $\mathcal A(\alpha, \beta, \gamma) \to \mathbb A(\alpha, \beta, \gamma)$
that sends $A \mapsto A$ and $B \mapsto B$.
\end{lemma}

\noindent As we will see, the map in Lemma \ref{lem:natH} might or might not be injective, depending on the choice of $\alpha, \beta, \gamma$.

\begin{lemma} There exists an algebra homomorphism $\mathbb A \to \mathbb F$ that sends
\begin{align*}
A \mapsto 0, \qquad \qquad B\mapsto 0, \qquad \qquad C \mapsto 0.
\end{align*}
\end{lemma}

\begin{lemma} There exists an automorphism $\rho$ of $\mathbb A$ that sends 
\begin{align*}
A \mapsto B, \qquad \qquad 
B \mapsto C, \qquad \qquad
C \mapsto A.
\end{align*}
Moreover $\rho^3=1$.
\end{lemma}

\begin{lemma} \label{lem:sign} There exists an automorphism $\zeta$ of $\mathbb A$ that sends 
\begin{align*}
A \mapsto -A, \qquad \qquad 
B \mapsto -B, \qquad \qquad
C \mapsto -C.
\end{align*}
Moreover $\zeta^2=1$.
\end{lemma}

\begin{lemma} For each standard generator $X$ of $\mathbb A$, there exists an  antiautomorphism $\sigma_X $  of $\mathbb A$ that fixes $X$ and
swaps the other two standard generators of $\mathbb A$.
Moreover $\sigma_X^2 = 1$.
\end{lemma}

\begin{lemma}  We have 
\begin{align*}
&\sigma_A \sigma_B = \sigma_B \sigma_C=\sigma_C \sigma_A = \rho, \qquad \qquad 
\sigma_B \sigma_A = \sigma_C\sigma_B = \sigma_A \sigma_C= \rho^{-1}
\end{align*}
and
\begin{align*}
\rho \sigma_A = \sigma_B \rho = \sigma_C, \qquad \qquad 
\rho \sigma_B = \sigma_C \rho = \sigma_A, \qquad \qquad 
\rho \sigma_C = \sigma_A \rho = \sigma_B.
\end{align*}
\end{lemma}

\begin{lemma} 
The map $\zeta$ commutes with each of $\sigma_A, \sigma_B, \sigma_C, \rho$.
\end{lemma}

\begin{lemma} For $\mathbb A$ we have
\begin{align*}
\lbrack AB, BA\rbrack=0, \qquad \quad \lbrack BC, CB\rbrack=0, \qquad \quad \lbrack CA, AC\rbrack=0.
\end{align*}
\end{lemma}
\begin{proof} The first equation follows from
\begin{align*}
ABBA-BAAB = A(\alpha BAB+ \beta ABB+\gamma B)- (\alpha ABA+\beta AAB+\gamma A)B = 0.
\end{align*}
The other equations are similarly obtained.
\end{proof}

\section{Adjusting the parameters of $\mathbb A$}

\noindent We continue to discuss the algebra $\mathbb A(\alpha, \beta, \gamma)$ from Definition \ref{def:bbA}. We mention some symmetries
of this algebra that involve changing the parameters.
For the results in this section, the
proof is routine and omitted.

\begin{lemma} \label{lem:adjust} For $0 \not= \xi \in \mathbb F$ there exists an algebra isomorphism $\mathbb A (\alpha, \beta, \gamma) \to \mathbb A(\alpha, \beta, \xi^{-2}\gamma)$
that sends
\begin{align*} 
A \mapsto \xi A, \qquad \qquad B \mapsto \xi B, \qquad \qquad C \mapsto \xi C.
\end{align*}
\end{lemma}

\begin{lemma} \label{lem:adjust2} Assume that $\beta \not=0$. Then there exists an algebra isomorphism 
\begin{align*}
\mathbb A(\alpha, \beta, \gamma) \to \mathbb A(-\alpha \beta^{-1}, \beta^{-1}, -\gamma \beta^{-1})
\end{align*}
That sends 
\begin{align*}
A \mapsto B, \qquad \qquad B \mapsto A, \qquad \qquad C \mapsto C.
\end{align*}
\end{lemma}

\begin{lemma} \label{lem:adjust3} Assume that $\beta \not=0$. Then there exists an algebra antiisomorphism 
\begin{align*}
\mathbb A(\alpha, \beta, \gamma) \to \mathbb A(-\alpha \beta^{-1}, \beta^{-1}, -\gamma \beta^{-1})
\end{align*}
That sends 
\begin{align*}
A \mapsto A, \qquad \qquad B \mapsto B, \qquad \qquad C \mapsto C.
\end{align*}
\end{lemma}

\noindent We emphasize a special case of Lemmas \ref{lem:adjust2}, \ref{lem:adjust3}.

\begin{corollary} \label{cor:adjust} Assume that $\beta=-1$. Then there exists an automorphism of $\mathbb A (\alpha, \beta, \gamma)$
that sends
\begin{align*} 
A \mapsto B, \qquad \qquad B \mapsto A, \qquad \qquad C \mapsto C.
\end{align*}
Moreover, there exists an antiautomorphism of  $\mathbb A (\alpha, \beta, \gamma)$ that sends
\begin{align*}
A \mapsto A, \qquad \qquad B \mapsto B, \qquad \qquad C \mapsto C.
\end{align*}
\end{corollary}

\section{ A $\mathbb Z_2$-grading of $\mathbb A$}

\noindent We continue to discuss the algebra $\mathbb A = \mathbb A(\alpha, \beta, \gamma)$ from Definition \ref{def:bbA}.
We describe a $\mathbb Z_2$-grading of $\mathbb A$.

\begin{definition} \label{def:gr} \rm The subspaces $\mathbb A_0$, $\mathbb A_1$ of $\mathbb A$ are defined as follows. 
\begin{enumerate}
\item[\rm (i)] The subspace $\mathbb A_0$ is spanned by the products of standard generators 
\begin{align*}
 X_1 X_2 \cdots X_n, \qquad n \in \mathbb N, \qquad \mbox{\rm $n$ even}.
\end{align*}
\item[\rm (ii)] The subspace $\mathbb A_1$ is spanned by the products of standard generators 
\begin{align*}
 X_1 X_2 \cdots X_n, \qquad n \in \mathbb N, \qquad \mbox{\rm$n$ odd}.
\end{align*}
\end{enumerate}
\end{definition}

\begin{lemma} \label{lem:z2} The following {\rm (i)--(iii)} hold:
\begin{enumerate}
\item[\rm (i)] the sum $\mathbb A=\mathbb A_0+\mathbb A_1$ is direct;
\item[\rm (ii)] $1 \in \mathbb A_0$;
\item[\rm (iii)] 
$\mathbb A_0 \mathbb A_0 \subseteq \mathbb A_0, \qquad 
\mathbb A_0 \mathbb A_1 \subseteq \mathbb A_1, \qquad 
\mathbb A_1 \mathbb A_0 \subseteq \mathbb A_1, \qquad 
\mathbb A_1 \mathbb A_1 \subseteq \mathbb A_0$.
\end{enumerate}
\end{lemma}
\begin{proof} (i) We have $\mathbb A=\mathbb A_0+\mathbb A_1$  because $\mathbb A$ is spanned by the products of standard generators 
$ X_1 X_2 \cdots X_n$  $(n \in \mathbb N)$.
The sum  $\mathbb A=\mathbb A_0+\mathbb A_1$ is direct because for each relation in Definition \ref{def:bbA}, every term is the product of an odd number of standard generators.
\\
\noindent (ii) By Definition \ref{def:gr}. \\
\noindent (iii) By Definition \ref{def:gr}.
\end{proof}

\noindent By Lemma  \ref{lem:z2}, the algebra $\mathbb A$ has a $\mathbb Z_2$-grading $\mathbb A=\mathbb A_0+\mathbb A_1$.
Note that $\mathbb A_1$ contains $A,B,C$.
\medskip

\noindent The subspace $\mathbb A_0$ is a subalgebra of $\mathbb A$, which we now describe.

\begin{lemma}  \label{lem:A0} The subalgebra $\mathbb A_0$ is generated by 
\begin{align}\label{eq:list}
A^2, \quad B^2, \quad C^2, \quad AB, \quad BA, \quad BC, \quad CB, \quad CA, \quad AC.
\end{align}
\end{lemma}
\begin{proof} We refer to  Definition \ref{def:gr}(i).
 Pick an even $n \in \mathbb N$, and write $n=2r$.
A product of standard generators
$X_1X_2\cdots X_n$ is equal to $Y_1Y_2 \cdots Y_r$, where $Y_i = X_{2i-1} X_{2i}$ for $1 \leq i \leq r$. The element $Y_i$ is listed in
\eqref{eq:list}  for $1 \leq i \leq r$. The result follows.
\end{proof}

\noindent Next, we characterize  the  $\mathbb Z_2$-grading $\mathbb A=\mathbb A_0+\mathbb A_1$ using the automorphism $\zeta$ of $\mathbb A$
from Lemma \ref{lem:sign}.

\begin{lemma} We have 
\begin{align*}
\mathbb A_0 = \lbrace a \in \mathbb A\vert \zeta(a)=a\rbrace, \qquad \qquad \mathbb A_1 = \lbrace a \in \mathbb A \vert \zeta(a)=-a\rbrace.
\end{align*}
\end{lemma}
\begin{proof} For $n \in \mathbb N$ and a product of standard generators $X_1X_2 \cdots X_n$, we have
$\zeta(X_1X_2\cdots X_n)=(-1)^n X_1 X_2 \cdots X_n$. By this and Definition \ref{def:gr}, we have $\zeta(a)=a$ for $a \in \mathbb A_0$ and $\zeta(a)=-a$ for $a \in \mathbb A_1$.
The result follows from this and Lemma \ref{lem:z2}(i).
\end{proof}


\section{Some extreme cases}

\noindent We continue to discuss the algebra $\mathbb A(\alpha, \beta, \gamma)$ from Definition \ref{def:bbA}. 
In this section, we use two examples to illustrate how the structure of $\mathbb A(\alpha, \beta, \gamma)$ depends on $\alpha, \beta, \gamma$. We will compare the algebras
 $\mathbb A(0,0,\gamma)$ for $\gamma=0$ and $\gamma\not=0$.
 \medskip
 
 \noindent The algebra $\mathbb A(0,0,0)$ is defined by generators $A,B,C$ and relations
 \begin{align}
 \label{eq:rr1}
 BA^2=0, \qquad CB^2=0, \qquad  AC^2=0,  \\
  B^2A=0, \qquad C^2B=0, \qquad  A^2C=0.
  \label{eq:rr2}
 \end{align}

\begin{proposition} \label{thm:zzz} The algebra $\mathbb A(0,0,0)$ has
a basis consisting of the products of standard generators
\begin{align*}
X_1 X_2 \cdots X_n
\end{align*}
such that $n \in \mathbb N$ and 
for $1 \leq i \leq n-2$ the sequence $(X_i, X_{i+1},X_{i+2})$ is not one of
\begin{align*}
(B,A,A), \qquad (B,B,A), \qquad (C,B,B), \qquad (C,C,B), \qquad (A,C,C), \qquad (A,A,C).
\end{align*}
\end{proposition} 
\begin{proof} We invoke the Bergman diamond lemma \cite[Theorem~1.2]{berg}. The reduction rules from the diamond lemma
are given by \eqref{eq:rr1}, \eqref{eq:rr2}. For these reduction rules the overlap ambiguities are resolvable. 
For instance, the overlap ambiguity $(BA^2)C=B(A^2C)$ is resolvable because $(BA^2)C=0C=0$ and $B(A^2C)=B0=0$.
The remaining overlap ambiguities are resolvable in a similar way. The result follows from the diamond lemma.
\end{proof}

\begin{corollary} The algebra homomorphism $\mathcal A(0,0,0)\to \mathbb A(0,0,0)$ from Lemma \ref{lem:natH}
is injective.
\end{corollary}
\begin{proof} For $i,j,k\in \mathbb N$ the product $A^i (BA)^j B^k$ is included in the basis for
$\mathbb A(0,0,0)$ from Proposition \ref{thm:zzz}. The result follows in view of
 Lemma  \ref{lem:basis}.
\end{proof}

 \noindent For the rest of this section, we consider $\mathbb A(0,0,\gamma)$ with $\gamma \not=0$.
 This algebra is defined by generators $A,B,C$ and relations
 \begin{align} \label{eq:3one}
 BA^2=\gamma A, \qquad CB^2=\gamma B,  \qquad AC^2= \gamma C, \\
 B^2A=\gamma B, \qquad C^2B=\gamma C, \qquad  A^2C=\gamma A. \label{eq:3two}
 \end{align}
\begin{definition} \label{def:dim3}  \rm Define the algebra $\mathbb S(\gamma) $ with one generator $D$ and one relation $D^3=\gamma D$. Note that $\lbrace D^i \rbrace_{i=0}^2$
is a basis for $\mathbb S(\gamma)$.
\end{definition}
\begin{proposition} \label{thm:three} For $0 \not=\gamma \in \mathbb F$ there exists an algebra isomorphism $\mathbb A(0,0,\gamma) \to \mathbb S(\gamma) $ that sends
\begin{align*}
A \mapsto D, \qquad \qquad B \mapsto D, \qquad \qquad C\mapsto D.
\end{align*}
\end{proposition} 
\begin{proof} By the nature of the relations \eqref{eq:3one}, \eqref{eq:3two}  there exists an algebra
homomorphism $\mathbb A(0,0,\gamma) \to \mathbb S(\gamma)$ that sends each of $A,B,C$ to $D$.
The homomorphism
is surjective by construction. We show that the homomorphism is injective. To do this, we show that $A=B=C$ and $A^3=\gamma A$.
We have
\begin{align*}
ACB = \gamma^{-1} (A^2C)CB = \gamma^{-1} A^2(C^2B) = A^2C = \gamma A,
\end{align*}
and also
\begin{align*}
ACB= \gamma^{-1} AC(CB^2) = \gamma^{-1} (AC^2)B^2= CB^2= \gamma B.
\end{align*}
Therefore $A=B$. Similarly we obtain $B=C$. Setting $A=B=C$ in \eqref{eq:3one} we obtain $A^3=\gamma A$.
The result follows.
\end{proof}

\begin{corollary} For $0 \not=\gamma \in \mathbb F$ the algebra homomorphism $\mathcal A(0,0,\gamma )\to \mathbb A(0,0,\gamma)$ from Lemma \ref{lem:natH}
is not injective.
 \end{corollary}
 \begin{proof} By Corollary 
 \ref{cor:NC}, the
  dimension of $\mathcal A(0,0,\gamma)$ is infinite. By Definition \ref{def:dim3} and Proposition \ref{thm:three}, the dimension of $\mathbb A(0,0,\gamma )$ is equal to 3.
 \end{proof}



\section{Some algebra homomorphisms that involve $\mathbb A$}

\noindent  Recall the algebra $\mathbb A=\mathbb A(\alpha, \beta, \gamma)$ from Definition \ref{def:bbA}. In this section, we describe some algebra homomorphisms that involve $\mathbb A$.
 These homomorphisms illuminate various aspects
of $\mathbb A$, and are useful in computations.

\begin{lemma} \label{lem:surj} Assume that $\gamma=0$. Then there exists an algebra homomorphism $\mathbb A (\alpha, \beta, \gamma) \to \mathcal A(\alpha, \beta, \gamma)$
that sends
\begin{align*}
A \mapsto A, \qquad \qquad B\mapsto B, \qquad \qquad C \mapsto 0.
\end{align*}
\end{lemma}
\begin{proof} The given map respects the down-up relations.
\end{proof}

\begin{corollary} \label{lem:gamZ} Assume that $\gamma=0$. Then  the algebra homomorphism $\mathcal A(\alpha, \beta, \gamma)\to \mathbb A(\alpha, \beta, \gamma)$ from Lemma \ref{lem:natH}
is injective.
\end{corollary}
\begin{proof} Consider the composition $\mathcal A(\alpha, \beta, \gamma) \to \mathbb A(\alpha, \beta, \gamma) \to \mathcal A(\alpha, \beta, \gamma)$, where the left factor is from Lemma
\ref{lem:natH} and the right factor is from Lemma \ref{lem:surj}. This composition is the identity map on $\mathcal A(\alpha, \beta, \gamma)$.
Therefore, the left factor is injective.
\end{proof}

\begin{lemma} \label{lem:ANZ} Assume that $\alpha \not=0$.
Then there exists an algebra homomorphism $\mathbb A(\alpha, \beta, \gamma) \to {\rm Mat}_3(\mathbb F) \otimes \mathbb F\lbrack t, t^{-1}\rbrack$ that sends
\begin{align*}
A \mapsto \begin{pmatrix} 0 & t  & 0 \\
     0 &0& 0 \\
     0 & -\frac{ \gamma}{\alpha t} & 0
     \end{pmatrix}, 
     \quad 
    B \mapsto  \begin{pmatrix} 0 & 0 & - \frac{\gamma}{\alpha t} \\
    0 & 0 &  t  \\
     0 &0& 0 
     \end{pmatrix}, 
     \quad 
 C\mapsto  \begin{pmatrix} 0 & 0 & 0 \\  -\frac{\gamma}{\alpha t}  & 0 & 0\\
    t & 0 & 0
     \end{pmatrix}.
\end{align*}
\end{lemma}
\begin{proof} It is readily checked that the given matrices satisfy the $\mathbb Z_3$-symmetric down-up relations with parameters $\alpha, \beta, \gamma$.
\end{proof}


\noindent Our next general goal is to describe how the algebra $\mathbb A$ is related to various other algebras.

\section{How $\mathbb A$ is related to the Weyl algebra}

\noindent  We continue to discuss the algebra $\mathbb A = \mathbb A(\alpha, \beta, \gamma)$ from Definition \ref{def:bbA}. In this section,
we describe how $\mathbb A$ is related to the Weyl algebra \cite{gaddis}.

\begin{definition} \rm (See \cite[Section~1]{gaddis}.)  For $\theta \in \mathbb F$, define the algebra $W=W(\theta)$ by generators $A,B$ and the relation $AB-BA=\theta $.
We call $W$ the {\it Weyl algebra} with parameter $\theta$.
\end{definition}
\noindent We remark that the elements $\lbrace A^i B^j \vert i,j \in \mathbb N\rbrace$ form a basis for $W(\theta)$; see for example \cite{dixmier}.
\begin{lemma} \label{lem:C} Referring to the Weyl algebra $W(\theta)$, define $C=-A-B$. Then
\begin{align*}
AB-BA=\theta , \qquad \quad BC-CB=\theta , \qquad \quad CA-AC=\theta .
\end{align*}
\end{lemma} 
\begin{proof} This is routinely checked.
\end{proof}

\begin{lemma} \label{lem:Wdu} Let $\theta \in \mathbb F$ and consider the Weyl algebra $W(\theta)$. For  $\xi \in \mathbb F$, the elements $A,B,C$ of $W(\theta)$ satisfy the
$\mathbb Z_3$-symmetric down-up relations with parameters
\begin{align*}
\alpha = \xi+1, \qquad \quad \beta =-\xi, \qquad \quad \gamma=(\xi-1)\theta.
\end{align*}
\end{lemma}
\begin{proof} We have
\begin{align*}
&B^2 A - \alpha BAB -\beta  AB^2 -\gamma B \\
& \quad = \xi\bigl(AB-BA-\theta \bigr)B-B\bigl(AB-BA-\theta ) = 0.
\end{align*}
\noindent We also have
\begin{align*}
& BA^2 - \alpha  ABA -\beta A^2 B -\gamma A \\
&\quad = \xi A\bigl(AB-BA-\theta \bigr) - \bigl(AB-BA-\theta \bigr) A = 0. \\
\end{align*}
The result follows by these comments and $\mathbb Z_3$-symmetry.
\end{proof}

\begin{theorem} \label{prop:W} Pick $\theta, \xi \in \mathbb F$ and consider  the algebra $\mathbb A=\mathbb A(\xi+1, -\xi, (\xi-1)\theta)$. There exists an algebra homomorphism
$\mathbb A \to W(\theta)$ that sends
\begin{align*}
A \mapsto A, \qquad \qquad B\mapsto B, \qquad \qquad C \mapsto C.
\end{align*}
\end{theorem} 
\begin{proof} By Lemma \ref{lem:Wdu}. 
\end{proof}

\noindent We have a comment about the elements $A,B,C$ of $W(\theta)$. By Lemma  \ref{lem:C}
we have $A+B+C=0$. However, this equation is not used
in our proof of Theorem \ref{prop:W}. This comment motivates
the following  $\mathbb Z_3$-symmetric generalization of $W(\theta)$.

\begin{definition} \label{def:bbW}  \rm For $\theta \in \mathbb F$, define the algebra $\mathbb W=\mathbb W(\theta)$ by generators $A,B, C$ and  relations
\begin{align*}
AB-BA=\theta , \qquad \quad BC-CB=\theta , \qquad \quad CA-AC=\theta .
\end{align*}
We call $\mathbb W$ the {\it $\mathbb Z_3$-symmetric Weyl algebra} with parameter $\theta$.
\end{definition} 

\begin{lemma} \label{lem:Wbasis} For $\theta \in \mathbb F$, the vector space $\mathbb W(\theta)$ has a basis
\begin{align*}
A^i B^j C^k \qquad \qquad i,j,k \in \mathbb N.
\end{align*}
\end{lemma}
\begin{proof} We invoke the Bergman diamond lemma \cite[Theorem~1.2]{berg}. With respect to lexicographical order (alphabetical order),
we have the reduction rules
\begin{align*}
BA=AB-\theta, \qquad \quad CB=BC-\theta , \qquad \quad CA=AC+ \theta .
\end{align*}
There is a unique overlap ambiguity, which is $(CB)A= C(BA)$. This ambiguity is resolvable, because
\begin{align*}
(CB)A &= (BC-\theta )A = B(CA)-\theta A = B(AC+\theta )-\theta A\\
& = (BA)C-\theta A+\theta B = (AB-\theta )C-\theta A+\theta B= ABC-\theta A+\theta B-\theta C
\end{align*}
and also
\begin{align*}
 C(BA) &= C(AB-\theta ) = (CA)B-\theta C= (AC+\theta )B-\theta C \\
&= A(CB)+\theta B-\theta C= A(BC-\theta )+\theta B-\theta C= ABC-\theta A+\theta B-\theta C.
\end{align*}
The result follows by the  diamond lemma.

\end{proof}

\begin{lemma} \label{lem:bbWdu} 
For  $\theta, \xi \in \mathbb F$, the elements $A,B,C$ of $\mathbb W(\theta)$ satisfy the
$\mathbb Z_3$-symmetric down-up relations with parameters $\alpha = \xi+1$, $\beta =-\xi$, $\gamma=(\xi-1)\theta $.
\end{lemma}
\begin{proof} Similar to the proof of Lemma \ref{lem:Wdu}.
\end{proof}

\begin{theorem} \label{prop:XiVal}
 Pick $\theta, \xi \in \mathbb F$ and consider  the algebra $\mathbb A=\mathbb A(\xi+1, -\xi, (\xi-1)\theta )$. There exists an algebra homomorphism
$\mathbb A \to \mathbb W(\theta)$ that sends
\begin{align*}
A \mapsto A, \qquad \qquad B\mapsto B, \qquad \qquad C \mapsto C.
\end{align*}
\end{theorem} 
\begin{proof} By Lemma \ref{lem:bbWdu}. 
\end{proof}

\section{The Lie algebra $\mathbb L(\gamma)$ and its relationship to $\mathbb A$}

We continue to discuss the algebra $\mathbb A=\mathbb A(\alpha, \beta, \gamma)$ from Definition \ref{def:bbA}. In this section, we assume that
 $\alpha=2$ and $\beta=-1$. In this case the relations in Definition \ref{def:bbA}  become
\begin{align*}
&\lbrack A, \lbrack A, B \rbrack \rbrack = \gamma A, \qquad \qquad \lbrack B, \lbrack B, A \rbrack \rbrack = \gamma B, \\
&\lbrack B, \lbrack B, C \rbrack \rbrack =  \gamma B, \qquad \qquad \lbrack C, \lbrack C, B \rbrack \rbrack = \gamma C, \\
&\lbrack C, \lbrack C, A \rbrack \rbrack = \gamma C, \qquad \qquad \lbrack A, \lbrack A, C \rbrack \rbrack = \gamma A,
\end{align*}
where $\lbrack x,y\rbrack = xy-yx$. This motivates the following definition.

\begin{definition} \label{def:Lie} \rm For $\gamma \in \mathbb F$ the Lie algebra $\mathbb L = \mathbb L(\gamma)$ is defined by generators $A,B,C$ 
and the following relations:
\begin{align}
&\lbrack A, \lbrack A, B \rbrack \rbrack = \gamma A, \qquad \qquad \lbrack B, \lbrack B, A \rbrack \rbrack = \gamma B, \label{eq:L1} \\
&\lbrack B, \lbrack B, C \rbrack \rbrack =  \gamma B, \qquad \qquad \lbrack C, \lbrack C, B \rbrack \rbrack = \gamma C,   \label{eq:L2} \\
&\lbrack C, \lbrack C, A \rbrack \rbrack = \gamma C, \qquad \qquad \lbrack A, \lbrack A, C \rbrack \rbrack = \gamma A.  \label{eq:L3}
\end{align}
We call $\mathbb L$ the {\it $\mathbb Z_3$-symmetric down-up Lie algebra} with parameter $\gamma$. We call $A,B,C$ the {\it standard generators} of $\mathbb L$.
We call the relations  \eqref{eq:L1}--\eqref{eq:L3}  the {\it $\mathbb Z_3$-symmetric down-up Lie algebra relations} with parameter $\gamma$.
\end{definition}

\noindent  For a Lie algebra $\mathcal L$, its enveloping algebra $U(\mathcal L)$ is constructed in \cite[Section~9.1]{carter}. 
We will consider  $U(\mathbb L)$.

\begin{proposition} \label{lem:ALiso} Pick $\gamma \in \mathbb F$. Consider the algebra $\mathbb A = \mathbb A(2,-1, \gamma)$ and the Lie algebra $\mathbb L = \mathbb L(\gamma)$.
There exists an algebra isomorphism $\mathbb A \to U(\mathbb L)$ that sends
\begin{align*}
A \mapsto A, \qquad \qquad B \mapsto B, \qquad \qquad C \mapsto C.
\end{align*}
\end{proposition}
\begin{proof} By the enveloping algebra construction \cite[p.~153]{carter}.
\end{proof}

\noindent We will say more about $\mathbb L(\gamma)$ later in the paper.

\section{The reduced $\mathbb Z_3$-symmetric down-up algebra} 

\noindent We continue to discuss the algebra $\mathbb A = \mathbb A(\alpha, \beta, \gamma)$  from Definition \ref{def:bbA}. 
\medskip

\noindent Throughout 
 this section, we assume $\alpha \not=0$.
 \medskip
 
 \noindent
For $\theta \in \mathbb F$ we introduce an algebra $\mathbb R(\theta)$, and explain how it is related to $\mathbb A$.

\begin{definition}\rm For $\theta \in \mathbb F$, the algebra $\mathbb R(\theta)$ is defined by generators $A,B,C$ and the following relations:
\begin{align}
&A^2=0, \qquad \qquad B^2 =0, \qquad \qquad C^2=0, \label{eq:aa}\\
& ABA=\theta A, \qquad BCB=\theta B, \qquad CAC=\theta C, \label{eq:aba} \\
& BAB = \theta B, \qquad CBC=\theta C, \qquad ACA=\theta A. \label{eq:bab}
\end{align}
We call
$\mathbb R(\theta)$ the {\it reduced $\mathbb Z_3$-symmetric down-up algebra} with parameter $\theta$.
We call $A,B,C$ the {\it standard generators} for $\mathbb R(\theta)$. We call the relations \eqref{eq:aa}--\eqref{eq:bab} the {\it reduced 
$\mathbb Z_3$-symmetric down-up relations} with parameter $\theta$.
\end{definition}

\begin{proposition} \label{prop:RT} Let $\alpha, \beta, \gamma \in \mathbb F$ with $\alpha \not=0$. Then there exists an algebra homomorphism
$\mathbb A(\alpha, \beta, \gamma) \to \mathbb R(-\alpha^{-1} \gamma)$ that sends
\begin{align*}
A \mapsto A, \qquad \qquad B \mapsto B, \qquad \qquad C \mapsto C.
\end{align*}
\end{proposition}
\begin{proof} In the $\mathbb Z_3$-symmetric down-up relations \eqref{eq:z3du1}--\eqref{eq:z3du3}, set $A^2=B^2=C^2=0$ and
rearrange terms to get \eqref{eq:aba} and \eqref{eq:bab}, where $\theta=-\alpha^{-1} \gamma$. The result follows.
\end{proof} 
\noindent  For the rest of this section, we fix $\theta \in \mathbb F$.
\medskip

\noindent
Our next goal is to display a basis for $\mathbb R(\theta)$.
\begin{proposition}\label{prop:RTbasis}
 The algebra $\mathbb R(\theta)$ has
a basis consisting of the products of standard generators
\begin{align*}
X_1 X_2 \cdots X_n
\end{align*}
such that $n \in \mathbb N$ and $X_{i-1} \not=X_i$ for $2 \leq i \leq n$
and  $ X_{i-1} \not=X_{i+1} $  for $2 \leq i \leq n-1$.
\end{proposition} 
\begin{proof} We invoke the Bergman diamond lemma \cite[Theorem~1.2]{berg}. The relations
\eqref{eq:aa}--\eqref{eq:bab} give reduction rules with respect to the word-length partial order.
The overlap ambiguities are resolvable. For instance, the overlap ambiguity $(ABA)CA=AB(ACA)$ is
resolvable, because
\begin{align*}
(ABA)CA= \theta ACA= \theta^2 A, \qquad \qquad AB(ACA) = \theta ABA = \theta^2 A.
\end{align*}
The overlap ambiguity $(AA)BA= A(ABA)$ is resolvable, because
\begin{align*}
(AA)BA = 0 BA = 0, \qquad \qquad A(ABA) = \theta A^2 = \theta 0 = 0.
\end{align*}
The overlap ambiguity $(ABA)A= AB(A^2)$ is resolvable, because
\begin{align*}
(ABA)A= \theta A^2 =\theta 0 = 0, \qquad \qquad AB(A^2) = AB 0  = 0.
\end{align*}
The remaining overlap ambiguities are  resolvable in a similar way.
The result follows by the  diamond lemma.
\end{proof}

\noindent In Proposition \ref{prop:RTbasis} we gave a basis for $\mathbb R(\theta)$. We now describe this basis in more detail. Let $\rho$ denote the automorphism
of $\mathbb R(\theta)$ that sends $A \mapsto B \mapsto C \mapsto A$.

\begin{definition}\rm Let $G$ denote a standard generator of $\mathbb R(\theta)$.
 For $n \geq 2$ define
 $G^+_n = X_1 X_2 \cdots X_n$, where  $\lbrace X_i \rbrace_{i=1}^n $ are standard generators of $\mathbb R(\theta)$ such that
 $X_1=G$ and $X_i = \rho(X_{i-1})$ for $2 \leq i \leq n$.
 Similarly define $G^-_i = X_1 X_2 \cdots X_n$, where  $\lbrace X_i \rbrace_{i=1}^n $  are standard generators of $\mathbb R(\theta)$ such that
 $X_1=G$ and $X_i = \rho^{-1}(X_{i-1})$ for $2 \leq i \leq n$.
\end{definition}

\begin{example}\rm For $2 \leq n \leq 6$ the elements $A^\pm _n, B^\pm_n, C^\pm_n$ of $\mathbb R(\theta)$ are given in the
tables below:
 \bigskip

\centerline{
\begin{tabular}[t]{c|ccc}
 $n$ & $A^+_n$ & $B^+_n$ & $C^+_n$ 
 \\
 \hline
 2 &  $AB$&$BC$ & $CA$
   \\
 3 & $ABC$  & $BCA$ & $CAB$ 
   \\
 4 & $ABCA$ & $BCAB$ & $CABC$
  \\
 5 & $ABCAB$ & $BCABC$ & $CABCA$
  \\
 6 & $ABCABC$ & $BCABCA$ & $CABCAB$
   \end{tabular}}  
     \bigskip
     
     \centerline{
\begin{tabular}[t]{c|ccc}
 $n$ & $A^-_n$ & $B^-_n$ & $C^-_n$ 
 \\
 \hline
 2 &  $AC$&$BA$ & $CB$
   \\
 3 & $ACB$  & $BAC$ & $CBA$ 
   \\
 4 & $ACBA$ & $BACB$ & $CBAC$
  \\
 5 & $ACBAC$ & $BACBA$ & $CBACB$
  \\
 6 & $ACBACB$ & $BACBAC$ & $CBACBA$
   \end{tabular}}  
     \bigskip

\end{example}

\begin{lemma} Referring to Proposition \ref{prop:RTbasis}, the given basis for $\mathbb R(\theta)$ consists of
\begin{align*}
&1, \quad A, \quad B, \quad C, \\
& A^+_n, \quad B^+_n, \quad C^+_n, \quad A^-_n, \quad B^-_n, \quad C^-_n \qquad \qquad n\geq 2.
\end{align*}
\end{lemma}
\begin{proof} This is readily checked.
\end{proof}

\noindent Next, we consider $\mathbb R(\theta)$ from a Lie algebra point of view.
By \cite[Example~1.3]{carter}, any algebra $\Delta $ can be turned into a Lie algebra with Lie bracket $\lbrack x,y \rbrack = xy-yx$. The resulting
Lie algebra will be denoted by $\lbrack \Delta \rbrack$. Let us consider $\lbrack \mathbb R(\theta) \rbrack$.

\begin{lemma} \label{lem:turnLie} The following relations hold in the Lie algebra  $\lbrack \mathbb R(\theta) \rbrack$:
\begin{align*}
&\lbrack A, \lbrack A, B \rbrack \rbrack = -2 \theta A, \qquad \qquad \lbrack B, \lbrack B, A \rbrack \rbrack = -2 \theta B,  \\
&\lbrack B, \lbrack B, C \rbrack \rbrack =  -2 \theta B, \qquad \qquad \lbrack C, \lbrack C, B \rbrack \rbrack = -2 \theta C,   \\
&\lbrack C, \lbrack C, A \rbrack \rbrack = -2\theta  C, \qquad \qquad \lbrack A, \lbrack A, C \rbrack \rbrack = -2 \theta A.  
\end{align*}
\end{lemma}
\begin{proof} The first relation holds, because
\begin{align*}
\lbrack A, \lbrack A, B \rbrack \rbrack = A^2 B - 2 ABA + BA^2 = -2 ABA = -2 \theta A.
\end{align*}
The remaining relations are similarly checked.
\end{proof} 
\noindent Recall the Lie algebra $\mathbb L(\gamma)$ from Definition \ref{def:Lie}.
\begin{proposition} \label{lem:LR} There exists a Lie algebra homomorphism $\mathbb L(-2\theta) \to \lbrack \mathbb R(\theta) \rbrack$
that sends
\begin{align*}
A \mapsto A, \qquad \qquad B \mapsto B, \qquad \qquad C \mapsto C.
\end{align*}
\end{proposition}
\begin{proof} Compare Definition \ref{def:Lie}
and
Lemma \ref{lem:turnLie}.
\end{proof}


\section{How  $\mathbb A$ is related to the Lie algebra $\mathfrak{sl}_2$}

\noindent Recall  the algebra $\mathbb A=\mathbb A(\alpha, \beta, \gamma)$ from Definition \ref{def:bbA}.
In this section, we describe how $\mathbb A$ is related to the Lie algebra $\mathfrak{sl}_2$.
\medskip


\noindent For $n\geq 2$ the Lie algebra $\mathfrak{sl}_n$ consists of the matrices in ${\rm Mat}_n(\mathbb F)$ that have trace 0, together with the Lie bracket
$\lbrack x,y \rbrack=xy-yx$. The dimension of the vector space $\mathfrak{sl}_n$ is equal to $n^2-1$. 

\begin{definition}\label{def:ABCsl2} We define elements $A,B,C$ in $\mathfrak{sl}_2$ as follows:
\begin{align*}
A=\begin{pmatrix} 1 &-1 \\1 &-1  \end{pmatrix}, \qquad \quad
B=\begin{pmatrix} 0&0 \\1 &0  \end{pmatrix}, \qquad \quad
C=\begin{pmatrix} 0&-1 \\0 &0 \end{pmatrix}.
\end{align*}
\end{definition}

\begin{lemma} \label{lem:ABCrel} The elements $A,B,C$ form a basis for $\mathfrak{sl}_2$. Moreover
\begin{align*}
\lbrack A,B \rbrack = C-A-B,\qquad 
\lbrack B,C\rbrack = A-B-C,\qquad 
\lbrack C,A\rbrack = B-C-A.
\end{align*}
\end{lemma}
\begin{proof} The first assertion follows from Definition \ref{def:ABCsl2}. The three relations are checked by matrix
multiplication.
\end{proof}

\begin{proposition} \label{lem:AABsl2}
The $\mathfrak{sl}_2$ basis elements $A,B,C$ satisfy the following relations:
\begin{align*}
&\lbrack A, \lbrack A, B \rbrack \rbrack = 2A, \qquad \qquad \lbrack B, \lbrack B, A \rbrack \rbrack = 2B, \\
&\lbrack B, \lbrack B, C \rbrack \rbrack =  2B, \qquad \qquad \lbrack C, \lbrack C, B \rbrack \rbrack = 2C, \\
&\lbrack C, \lbrack C, A \rbrack \rbrack = 2C, \qquad \qquad \lbrack A, \lbrack A, C \rbrack \rbrack = 2A.
\end{align*}
\end{proposition} 
\begin{proof} Use the relations in Lemma \ref{lem:ABCrel}.
\end{proof}

\begin{proposition} \label{lem:sl2mapL} There exists a Lie algebra homomorphism $ \mathbb L(2) \to \mathfrak{sl}_2$ that sends
\begin{align*}
A \mapsto A, \qquad \qquad B \mapsto B, \qquad \qquad C \mapsto C.
\end{align*}
\end{proposition}
\begin{proof}  The relations in Proposition \ref{lem:AABsl2} are the $\mathbb Z_3$-symmetric down-up Lie algebra relations with parameter $\gamma=2$.
\end{proof}

\begin{theorem} \label{lem:sl2map} For the algebra $\mathbb A = \mathbb A(2, -1,2)$ there exists an algebra homomorphism $\mathbb A \to U(\mathfrak{sl}_2)$ that sends
\begin{align*}
A \mapsto A, \qquad \qquad B \mapsto B, \qquad \qquad C \mapsto C.
\end{align*}
\end{theorem}
\begin{proof} Abbreviate $\mathbb L=\mathbb L(2)$. The Lie algebra homomorphism $ \mathbb L \to \mathfrak{sl}_2$ from Proposition   \ref{lem:sl2mapL} induces an algebra homomorphism
$\natural: U( \mathbb L) \to U(\mathfrak{sl}_2)$. The map in the theorem statement is the composition
 $\mathbb A \to U(\mathbb L) \to  U(\mathfrak{sl}_2)$, where the left factor is from
Proposition \ref{lem:ALiso}  and the right factor is equal to $\natural$.  
\end{proof}

\begin{note}\rm The above basis $A,B,C$ for $\mathfrak{sl}_2$ appears in the paper \cite{benkter} about the equitable basis for $\mathfrak{sl}_2$. 
The vector space  $\mathfrak{sl}_2$ is equipped with a bilinear form called the Killing form. The basis
$A,B,C$  and the equitable basis $x,y, z$ are dual with respect to a scalar multiple of the Killing form.
 In the notation of \cite[Eqn.~(2.11)]{benkter} we have $A=x^*$, $B=y^*$, $C=z^*$.
\end{note}

\section{How $\mathbb A$ is related to the Lie algebra $\mathfrak{sl}_3$}

We continue to discuss the algebra $\mathbb A=\mathbb A(\alpha, \beta, \gamma)$ from Definition \ref{def:bbA}. 
In this section, we describe how $\mathbb A$ is related to the Lie algebra $\mathfrak{sl}_3$.
\medskip

\noindent Throughout this section, fix $\xi \in \mathbb F$.

\begin{definition}\label{def:ABCsl3} We define elements $A,B,C$ in $\mathfrak{sl}_3$ as follows:
\begin{align*}
A=\begin{pmatrix} 0&1&0 \\0 &0&0 \\ 0&\xi&0 \end{pmatrix}, \qquad \quad
B=\begin{pmatrix} 0&0&\xi \\0 &0&1 \\ 0&0&0 \end{pmatrix}, \qquad \quad
C=\begin{pmatrix} 0&0&0 \\\xi &0&0 \\ 1&0&0 \end{pmatrix}.
\end{align*}
\end{definition}
\noindent Our next goal is to show that $A,B,C$ generate the Lie algebra $\mathfrak{sl}_3$, provided that $1+ \xi^3\not=0$.

\begin{lemma} \label{lem:six} The following relations hold in the Lie algebra $\mathfrak{sl}_3$:
\begin{align*}
\lbrack A, B\rbrack &=\begin{pmatrix} 0&-\xi^2&1 \\0 &-\xi&0 \\ 0&0&\xi \end{pmatrix}, \qquad \qquad \quad
\lbrack B, C\rbrack =\begin{pmatrix} \xi&0&0 \\1 &0&-\xi^2 \\ 0&0&-\xi \end{pmatrix}, \\
\lbrack C, A\rbrack &=\begin{pmatrix} -\xi&0&0 \\0 &\xi&0 \\ -\xi^2&1&0 \end{pmatrix}, \qquad \qquad 
\lbrack A,\lbrack  B, C\rbrack \rbrack =\begin{pmatrix} 1&-\xi&-\xi^2 \\0 &\xi^3-1&0 \\ \xi&\xi^2&-\xi^3 \end{pmatrix},\\
\lbrack B, \lbrack C, A\rbrack \rbrack &=\begin{pmatrix} -\xi^3&\xi&\xi^2 \\-\xi^2 &1&-\xi \\ 0&0&\xi^3-1 \end{pmatrix}, \qquad 
\lbrack C, \lbrack A, B\rbrack \rbrack =\begin{pmatrix} \xi^3-1&0&0 \\\xi^2 &-\xi^3&\xi \\ -\xi&-\xi^2&1 \end{pmatrix}.
\end{align*}
\end{lemma}
\begin{proof} Use Definition \ref{def:ABCsl3} and  matrix multiplication.
\end{proof}

\noindent In the above lemma, the last three matrices are linearly dependent. By inspection or the Jacobi identity,
\begin{align} \label{eq:jacobi}
\lbrack A,\lbrack  B, C\rbrack \rbrack +
\lbrack B, \lbrack C, A\rbrack \rbrack+
\lbrack C, \lbrack A, B\rbrack \rbrack =0.
\end{align}
For $1\leq i,j\leq 3$ define $E_{i,j} \in {\rm Mat}_3(\mathbb F)$  that has $(i,j)$-entry 1 and all other entries 0. Define $H_1 = E_{1,1}-E_{2,2}$ and $H_2 = E_{2,2}-E_{3,3}$.
The following is a basis for the vector space $\mathfrak{sl}_3$:
\begin{align*}
E_{1,2}, \quad E_{2,3}, \quad E_{3,1}, \quad E_{2,1}, \quad E_{3,2}, \quad E_{1,3}, \quad H_1,\quad H_2.
\end{align*}

\begin{lemma} \label{lem:eight}  Assume $1+\xi^3\not=0$. Then the following relations hold in the Lie algebra $\mathfrak{sl}_3$:
\begin{align*}
E_{1,2} &= \frac{(1+\xi^3) A -\xi^4 \lbrack A, B\rbrack-\xi \lbrack C, A \rbrack-\xi^2 \lbrack A, \lbrack B, C \rbrack \rbrack}{(1+\xi^3)^2}, \\
E_{2,3} &= \frac{(1+\xi^3) B -\xi^4 \lbrack B, C\rbrack-\xi \lbrack A, B \rbrack-\xi^2 \lbrack B, \lbrack C, A \rbrack \rbrack}{(1+\xi^3)^2}, \\
E_{3,1} &= \frac{(1+\xi^3) C -\xi^4 \lbrack C, A\rbrack-\xi \lbrack B, C \rbrack-\xi^2 \lbrack C, \lbrack A, B \rbrack \rbrack}{(1+\xi^3)^2}
\end{align*}
\noindent and
\begin{align*}
E_{2,1} &= \frac{\xi^2(1+\xi^3)C +\lbrack B,C\rbrack + \xi^3 \lbrack C,A\rbrack+\xi \lbrack C, \lbrack A,B \rbrack \rbrack }{(1+\xi^3)^2}, \\
E_{3,2} &= \frac{\xi^2(1+\xi^3)A +\lbrack C,A\rbrack + \xi^3 \lbrack A,B\rbrack+\xi \lbrack A, \lbrack B,C \rbrack \rbrack }{(1+\xi^3)^2}, \\
E_{1,3} &= \frac{\xi^2(1+\xi^3)B +\lbrack A,B\rbrack + \xi^3 \lbrack B,C\rbrack+\xi \lbrack B, \lbrack C,A \rbrack \rbrack }{(1+\xi^3)^2}
\end{align*}
\noindent and
\begin{align*}
H_1 &=\frac{\xi A - \xi C }{1+\xi^3}+
 \frac{\xi^2 \lbrack A, B\rbrack+\xi^2\lbrack B, C\rbrack-2\xi^2\lbrack C, A\rbrack -\lbrack B, \lbrack C,A \rbrack \rbrack + (\xi^3-1)\lbrack C, \lbrack A, B \rbrack \rbrack   }{(1+\xi^3)^2}, \\
 H_2 &=\frac{\xi B - \xi A }{1+\xi^3}+
 \frac{\xi^2 \lbrack B, C\rbrack+\xi^2\lbrack C, A\rbrack-2\xi^2\lbrack A, B\rbrack -\lbrack C, \lbrack A,B \rbrack \rbrack + (\xi^3-1)\lbrack A, \lbrack B, C \rbrack \rbrack   }{(1+\xi^3)^2}.
\end{align*}
\end{lemma}
\begin{proof} Use Definition \ref{def:ABCsl3} and Lemma \ref{lem:six}.
\end{proof}
\begin{lemma} \label{lem:ABCgen}  Assume that $1+\xi^3 \not=0$. Then the following is a basis for the vector space $\mathfrak{sl}_3$:
\begin{align*}
&A, \qquad \quad B, \qquad  \quad C,  \qquad \quad \lbrack A, B\rbrack, \qquad \quad \lbrack B,C\rbrack, \\
& \lbrack C, A\rbrack, \qquad \quad  \lbrack A, \lbrack B, C\rbrack \rbrack, \qquad \quad  \lbrack B, \lbrack C, A\rbrack \rbrack.
\end{align*}
Moreover, the elements $A, B,C$ generate the Lie algebra $\mathfrak{sl}_3$.
\end{lemma} 
\begin{proof} The first assertion follows from \eqref{eq:jacobi} and Lemma \ref{lem:eight}. The second assertion follows from the first.
\end{proof}

\begin{proposition} \label{lem:sl3Rel}The $\mathfrak{sl}_3$ elements $A,B,C$ satisfy the following relations:
\begin{align*}
&\lbrack A, \lbrack A, B \rbrack \rbrack = - 2\xi A, \qquad \qquad \lbrack B, \lbrack B, A \rbrack \rbrack = - 2\xi B, \\
&\lbrack B, \lbrack B, C \rbrack \rbrack = - 2\xi B, \qquad \qquad \lbrack C, \lbrack C, B \rbrack \rbrack = - 2\xi C, \\
&\lbrack C, \lbrack C, A \rbrack \rbrack = - 2\xi C, \qquad \qquad \lbrack A, \lbrack A, C \rbrack \rbrack = - 2\xi A.
\end{align*}
\end{proposition}
\begin{proof}
Use Lemma \ref{lem:ANZ}, or else Definition \ref{def:ABCsl3} and Lemma \ref{lem:six}.
\end{proof}

\noindent Recall the Lie algebra $\mathbb L(\gamma)$ from Definition \ref{def:Lie}.
\begin{proposition} \label{lem:sl3mapL} There exists a Lie algebra homomorphism $ \mathbb L(-2\xi) \to \mathfrak{sl}_3$ that sends
\begin{align*}
A \mapsto A, \qquad \qquad B \mapsto B, \qquad \qquad C \mapsto C.
\end{align*}
This homomorphism is surjective, provided that $1+ \xi^3 \not=0$. 
\end{proposition}
\begin{proof}  The first assertion holds because the relations in Proposition \ref{lem:sl3Rel} are the $\mathbb Z_3$-symmetric down-up Lie algebra relations with parameter $\gamma=-2\xi $.
 The second assertion follows from  Lemma \ref{lem:ABCgen}. 
\end{proof}

\begin{theorem} \label{prop:nat} For the algebra $\mathbb A = \mathbb A(2,-1,-2\xi )$ there exists an algebra homomorphism $\mathbb A \to U(\mathfrak{sl}_3)$ that sends 
\begin{align*}
A \mapsto A, \qquad \qquad 
B \mapsto B, \qquad \qquad 
C \mapsto C.
\end{align*}
\noindent This homomorphism is surjective, provided that $1 + \xi^3 \not=0$. 
\end{theorem}
\begin{proof} By Propositions \ref{lem:ALiso}  and \ref{lem:sl3mapL}.
\end{proof}

\section{How $\mathbb A$ is related to the loop algebra $\mathfrak{sl}_3\otimes \mathbb F\lbrack t, t^{-1}\rbrack$}

\noindent 
We continue to discuss the algebra $\mathbb A=\mathbb A(\alpha, \beta, \gamma)$ from Definition \ref{def:bbA}. In this section, we describe how $\mathbb A$
is related to the $\mathfrak{sl}_3$ loop algebra. 
%

\begin{definition}\rm \label{def:loop}  (See \cite[Section~18.1]{carter}.) The vector space $\mathfrak{sl}_3 \otimes \mathbb F \lbrack t, t^{-1} \rbrack$ becomes a Lie algebra with
Lie bracket 
\begin{align*}
\lbrack x \otimes f, y \otimes g\rbrack = \lbrack x,y\rbrack \otimes fg \qquad \qquad x,y \in \mathfrak{sl}_3, \qquad f,g \in \mathbb F \lbrack t, t^{-1}\rbrack.
\end{align*}
The Lie algebra $\mathfrak{sl}_3 \otimes \mathbb F \lbrack t, t^{-1} \rbrack$ is called the {\it $\mathfrak{sl}_3$ loop algebra}.
\end{definition}

\begin{definition}\label{def:ABCsl3t} Let $\xi \in \mathbb F$. We define elements $A,B,C$ in $\mathfrak{sl}_3 \otimes \mathbb F \lbrack t, t^{-1} \rbrack$
 as follows:
\begin{align*}
A=\begin{pmatrix} 0&t&0 \\0 &0&0 \\ 0&\xi t^{-1}&0 \end{pmatrix}, \qquad \quad
B=\begin{pmatrix} 0&0&\xi t^{-1} \\0 &0&t \\ 0&0&0 \end{pmatrix}, \qquad \quad
C=\begin{pmatrix} 0&0&0 \\\xi t^{-1} &0&0 \\ t&0&0 \end{pmatrix}.
\end{align*}
\end{definition}

\begin{proposition} \label{lem:sl3Ret} The following relations hold in the Lie algebra $\mathfrak{sl}_3 \otimes \mathbb F \lbrack t, t^{-1} \rbrack$:
\begin{align*}
&\lbrack A, \lbrack A, B \rbrack \rbrack = - 2\xi A, \qquad \qquad \lbrack B, \lbrack B, A \rbrack \rbrack = - 2 \xi B, \\
&\lbrack B, \lbrack B, C \rbrack \rbrack = - 2 \xi B, \qquad \qquad \lbrack C, \lbrack C, B \rbrack \rbrack = - 2 \xi C, \\
&\lbrack C, \lbrack C, A \rbrack \rbrack = - 2 \xi C, \qquad \qquad \lbrack A, \lbrack A, C \rbrack \rbrack = - 2 \xi A.
\end{align*}
\end{proposition}
\begin{proof} Apply Lemma  \ref{lem:ANZ} with $\alpha=2$ and $\gamma=-2 \xi$.
\end{proof}

\noindent Recall the Lie algebra $\mathbb L(\gamma)$ from Definition \ref{def:Lie}.
\begin{proposition} \label{prop:natt} There exists a Lie algebra homomorphism
 $ \mathbb L(-2\xi)\to \mathfrak{sl}_3 \otimes \mathbb F \lbrack t, t^{-1} \rbrack$  that sends 
\begin{align*}
A \mapsto A, \qquad \qquad 
B \mapsto B, \qquad \qquad 
C \mapsto C.
\end{align*}
\end{proposition}
\begin{proof}  The relations in Proposition  \ref{lem:sl3Ret} are the $\mathbb Z_3$-symmetric down-up Lie algebra relations with parameter $\gamma=-2\xi $.
\end{proof}

\begin{theorem} \label{prop:2nat} For the algebra $\mathbb A = \mathbb A(2,-1,-2\xi )$,  there exists an algebra homomorphism $\mathbb A \to U\bigl(\mathfrak{sl}_3\otimes \mathbb F\lbrack t, t^{-1}\rbrack\bigr)$ that sends 
\begin{align*}
A \mapsto A, \qquad \qquad 
B \mapsto B, \qquad \qquad 
C \mapsto C.
\end{align*}
\end{theorem}
\begin{proof} By Propositions \ref{lem:ALiso}  and \ref{prop:natt}.
\end{proof}


\section{How $\mathbb A$ is related to the Kac-Moody Lie algebra $A^{(1)}_2$}

\noindent We continue to discuss the algebra $\mathbb A=\mathbb A(\alpha, \beta, \gamma)$ from Definition \ref{def:bbA}. 
In this section, we describe how $\mathbb A$ is related to the Kac-Moody Lie algebra $A^{(1)}_2$. Background information about
Kac-Moody Lie algebras can be found in \cite{carter,cp3,kac}.
\medskip

%
\noindent Consider the Cartan matrix $\mathcal C$ of type $A^{(1)}_2$:

\begin{align} \label{eq:Cartan}
{\mathcal C} = \begin{pmatrix} 2&-1 &-1\\
 -1& 2 &-1\\
 -1 &-1 &2
 \end{pmatrix}.
\end{align}

\begin{definition}\label{def:km} \rm (See \cite[pp.~xi, 54]{kac}). Define the Lie algebra $A^{(1)}_2$ by generators
\begin{align*}
e_i, \quad f_i, \quad h_i \qquad \qquad (1 \leq i \leq 3)
\end{align*}
and the following relations. For $1 \leq i,j\leq 3$,
\begin{align} \label{eq:km1}
& \lbrack h_i, h_j \rbrack = 0,  \qquad \qquad \quad \;\;
 \lbrack h_i, e_j \rbrack = {\mathcal C}_{i,j} e_j,  \\
 \label{eq:km2}
& \lbrack h_i, f_j \rbrack = -{\mathcal C}_{i,j} f_j, \qquad \quad 
 \lbrack e_i, f_j \rbrack = \delta_{i,j} h_i, \\
 \label{eq:km3}
& \lbrack e_i, \lbrack e_i, e_j \rbrack \rbrack =0,  \qquad \qquad \,
 \lbrack f_i, \lbrack f_i, f_j \rbrack \rbrack =0 \qquad \qquad (i \not=j).
\end{align}
We call $A^{(1)}_2$  the {\it Kac-Moody Lie algebra}  with {\it Cartan matrix $\mathcal C$}.
\end{definition}

\begin{remark}\rm What we call the Kac-Moody Lie algebra is often called the derived algebra. The derived algebra is denoted $L'$ in
\cite[p.~335]{carter} and 
$\mathfrak{g}'$ in \cite[p.~562]{cp3}, \cite[p.~xi]{kac}.
\end{remark}

\noindent The following basic facts about $A^{(1)}_2$ can be found in \cite{carter} or \cite{cp3} or \cite{kac}. The element $h_1+h_2+h_3$ is central in $A^{(1)}_2$. 
The elements $\lbrace h_i \rbrace_{i=1}^3$ form a basis for a commutative Lie subalgebra $H$ of $A^{(1)}_2$, called a Cartan
subalgebra. Let $N_+$ (resp. $N_-$) denote the Lie subalgebra of $A^{(1)}_2$ generated by $\lbrace e_i \rbrace_{i=1}^3$ (resp. $\lbrace f_i \rbrace_{i=1}^3$).
The Lie algebra $N_+$ (resp. $N_-$) is called the positive part (resp. negative part) of $A^{(1)}_2$.
The following sum is direct:
\begin{align*} A^{(1)}_2 = N_+ + H + N_-.
\end{align*}
The Lie algebra $N_+$ has a presentation by generators $\lbrace e_i \rbrace_{i=1}^3$
and relations
\begin{align}
 \lbrack e_i, \lbrack e_i, e_j \rbrack \rbrack =0  \qquad \qquad (1 \leq i,j\leq 3, \;i \not=j).
 \label{eq:SerreE}
\end{align}
The Lie algebra $N_-$ has a presentation by generators $\lbrace f_i \rbrace_{i=1}^3$
and relations
\begin{align*}
 \lbrack f_i, \lbrack f_i, f_j \rbrack \rbrack =0  \qquad \qquad (1 \leq i,j\leq 3, \;i \not=j).
\end{align*}
\noindent We now give some results about $N_+$; similar results hold for $N_-$.
\begin{lemma}\label{lem:sumEF} There exists a Lie algebra isomorphsm $\mathbb L(0) \to N_+$ that sends
\begin{align*}
A \mapsto e_1, \qquad \qquad B\mapsto e_2, \qquad \qquad C\mapsto e_3.
\end{align*}
\end{lemma}
\begin{proof} The relations \eqref{eq:SerreE} are the $\mathbb Z_3$-symmetric down-up Lie algebra relations with parameter $\gamma=0$. 
\end{proof}

\begin{corollary} \label{lem:kmInj} There exists a Lie algebra homomorphism $ \mathbb L(0) \to A^{(1)}_2$ that sends
\begin{align*}
A \mapsto e_1, \qquad \qquad B\mapsto e_2, \qquad \qquad C\mapsto e_3.
\end{align*}
This homomorphism is injective.
\end{corollary}
\begin{proof}  By Lemma \ref{lem:sumEF}.
\end{proof}

\begin{corollary} \label{lem:kmInj2} There exists an algebra homomorphism $ \mathbb A(2, -1,0) \to U(A^{(1)}_2)$ that sends
\begin{align*}
A \mapsto e_1, \qquad \qquad B\mapsto e_2, \qquad \qquad C\mapsto e_3.
\end{align*}
This homomorphism is injective.
\end{corollary}
\begin{proof}  By  Proposition \ref{lem:ALiso} (with $\gamma=0$) and Corollary \ref{lem:kmInj}.
\end{proof}


\begin{definition}\label{def:ABCs} \rm  Let $\xi \in \mathbb F$. We  define some elements $A, B, C$ in $A^{(1)}_2$ as follows: 
\begin{align*}
A = e_1   +\xi  f_2, \qquad \qquad B= e_2 +\xi f_3, \qquad \qquad C= e_3 +\xi f_1.
\end{align*}
\end{definition}
\begin{proposition} \label{lem:factA2}  With reference to Definition \ref{def:ABCs}, the following relations hold in $A^{(1)}_2$:
\begin{align*}
&\lbrack A, \lbrack A, B \rbrack \rbrack = -2 \xi A, \qquad \qquad \lbrack B, \lbrack B, A \rbrack \rbrack = -2 \xi  B, \\
&\lbrack B, \lbrack B, C \rbrack \rbrack =  -2 \xi B, \qquad \qquad \lbrack C, \lbrack C, B \rbrack \rbrack = -2 \xi C, \\
&\lbrack C, \lbrack C, A \rbrack \rbrack = -2 \xi C, \qquad \qquad \lbrack A, \lbrack A, C \rbrack \rbrack = -2 \xi A.
\end{align*}
\end{proposition}
\begin{proof} This is routinely checked using the relations  \eqref{eq:km1}--\eqref{eq:km3}.
\end{proof}

\begin{proposition}\label{prop:efA} With reference to Definition \ref{def:ABCs}, there exists a Lie algebra homomorphism $ \mathbb L(-2\xi ) \to A^{(1)}_2$ that sends
\begin{align*}
A \mapsto A, \qquad \qquad B \mapsto B, \qquad \qquad C \mapsto C.
\end{align*}
\end{proposition}
\begin{proof} The relations in Proposition \ref{lem:factA2} are the $\mathbb Z_3$-symmetric down-up Lie algebra relations with parameter $\gamma=-2\xi $.
\end{proof}

\begin{theorem}\label{cor:AandKM} Pick $\xi \in \mathbb F$ and consider the algebra $\mathbb A = \mathbb A(2, -1, -2\xi )$.
There exists an algebra homomorphism $ \mathbb A \to U(A^{(1)}_2)$ that sends
\begin{align*}
A \mapsto e_1   +\xi f_2, \qquad \qquad B \mapsto  e_2 +\xi f_3, \qquad \qquad C= \mapsto e_3 +\xi f_1.
\end{align*}
\end{theorem}
\begin{proof} By Propositions \ref{lem:ALiso}, \ref{prop:efA}.
\end{proof}

\section{How $\mathbb A$ is related to the $q$-Weyl algebra}

\noindent We continue to discuss the algebra $\mathbb A = \mathbb A(\alpha, \beta, \gamma)$ from Definition
 \ref{def:bbA}. In this section,
we describe how $\mathbb A$ is related to the $q$-Weyl algebra.
\medskip

\noindent
Throughout this section, we fix a nonzero $q \in \mathbb F$.
\medskip

\noindent For $\theta   \in \mathbb F$ the $q$-Weyl algebra $W_q(\theta)$ is defined by generators $A,B$ and the relation $AB-qBA=\theta $; see for example
\cite{gaddis}.
We consider the following $\mathbb Z_3$-symmetric generalization of $W_q(\theta)$.

\begin{definition} \label{def:QW} \rm For $\theta  \in \mathbb F$, define the algebra $\mathbb W_q=\mathbb W_q(\theta)$ by generators $A,B,C$ and the relations
\begin{align*}
AB-q BA=\theta ,\qquad \quad BC-qCB=\theta , \qquad\quad CA-qAC=\theta .
\end{align*}
We call $\mathbb W_q $ the {\it $\mathbb Z_3$-symmetric $q$-Weyl algebra} with parameter $\theta$.
\end{definition}

\begin{lemma} \label{lem:WbasisQ} For $\theta \in \mathbb F$ the vector space $\mathbb W_q(\theta)$ has a basis
\begin{align*}
A^i B^j C^k \qquad \qquad i,j,k \in \mathbb N.
\end{align*}
\end{lemma}
\begin{proof} We invoke the Bergman diamond lemma \cite[Theorem~1.2]{berg}. With respect to lexicographical order,
we have the reduction rules
\begin{align*}
BA=q^{-1} AB- q^{-1}\theta , \qquad  CB=q^{-1} BC- q^{-1}\theta , \qquad  CA=q AC+ \theta .
\end{align*}
There is a unique overlap ambiguity, which is $(CB)A= C(BA)$. This ambiguity is resolvable, because
\begin{align*}
(CB)A &= q^{-1}(BC-\theta )A = q^{-1}B(CA)- q^{-1}\theta A = q^{-1}B(qAC+\theta )-q^{-1}\theta A \\
&= (BA)C-q^{-1}\theta A+q^{-1}\theta B 
 = q^{-1}(AB-\theta )C-q^{-1}\theta A+q^{-1}\theta B\\
 &= q^{-1}ABC-q^{-1}\theta A+q^{-1}\theta B-q^{-1}\theta C
\end{align*}
and also
\begin{align*}
 C(BA) &= q^{-1}C(AB-\theta ) = q^{-1}(CA)B- q^{-1}\theta C= (AC+q^{-1}\theta )B-q^{-1}\theta C \\
&= A(CB)+q^{-1}\theta B-q^{-1}\theta C
= q^{-1}A(BC-\theta )+q^{-1}\theta B-q^{-1}\theta C \\
&= q^{-1}ABC-q^{-1}\theta A+q^{-1}\theta B-q^{-1}\theta C.
\end{align*}
The result follows by the  diamond lemma.

\end{proof}

\begin{proposition} \label{lem:bbWduQ} 
For  $\theta, \xi \in \mathbb F$ the elements $A,B,C$ of $\mathbb W_q(\theta)$ satisfy the
$\mathbb Z_3$-symmetric down-up relations with parameters 
\begin{align*}
\alpha = q \xi + q^{-1}, \qquad \quad \beta = -\xi, \qquad \quad \gamma = (\xi - q^{-1})\theta.
\end{align*}
\end{proposition}
\begin{proof} In the algebra $\mathbb W_q(\theta)$,
\begin{align*}
&B^2 A - \alpha BAB-\beta AB^2 - \gamma B \\
&\quad = \xi (AB-qBA-\theta ) B - q^{-1} B (AB-qBA-\theta )  = 0
\end{align*}
\noindent and
\begin{align*}
&BA^2 - \alpha ABA - \beta A^2 B - \gamma A \\
&\quad = \xi A(AB-qBA-\theta ) - q^{-1} (AB-qBA-\theta )A = 0.
\end{align*}
The result follows by these comments and $\mathbb Z_3$-symmetry.
\end{proof}

\begin{theorem} \label{prop:QWmap}
 Pick $\theta, \xi \in \mathbb F$ and define
 \begin{align*}
\alpha = q \xi + q^{-1}, \qquad \quad \beta = -\xi, \qquad \quad \gamma = (\xi - q^{-1})\theta.
\end{align*}
 There exists an algebra homomorphism
$\mathbb A(\alpha,\beta,\gamma) \to \mathbb W_q(\theta)$ that sends
\begin{align*}
A \mapsto A, \qquad \qquad B\mapsto B, \qquad \qquad C \mapsto C.
\end{align*}
\end{theorem} 
\begin{proof} By Proposition \ref{lem:bbWduQ}. 
\end{proof}

\section{How  $\mathbb A$ is related to $U_q(\mathfrak{sl}_2)$}

\noindent We continue to discuss the algebra $\mathbb A = \mathbb A(\alpha, \beta, \gamma)$ from Definition
 \ref{def:bbA}.
In this section, we describe how $\mathbb A$ is related to the quantized enveloping algebra  $U_q(\mathfrak{sl}_2)$.
We will work with the equitable presentation of $U_q(\mathfrak{sl}_2)$ \cite{equit}. Background information about this presentation can be found in
\cite{bocktingTer, funkNeub1, funkNeub2, equit, tersym, uawe, fduq, billiard, lrt, equitLu, wora}.
\medskip

\noindent Throughout this section, we fix a nonzero  $q \in \mathbb F$ such that $q^2 \not=1$.
\medskip

\noindent The following definition resembles Definition \ref{def:QW}; we include both definitions for the sake of completeness.

\begin{definition} \label{def:uq} \rm (See \cite[Theorem~2.1]{equit}.) 
Define the algebra $U_q(\mathfrak{sl}_2)$
by generators $x, y^{\pm 1}, z$ and relations $y y^{-1} = 1 = y^{-1} y$,
\begin{align*}
\frac{qxy-q^{-1} yx}{q-q^{-1}} = 1, \qquad \quad 
\frac{qyz-q^{-1} zy}{q-q^{-1}} = 1, \qquad \quad 
\frac{qzx-q^{-1} xz}{q-q^{-1}} = 1.
\end{align*}
\end{definition}

\begin{lemma} {\rm (See \cite[Lemma~10.7]{uawe}.)} The following is a basis for the vector space $U_q(\mathfrak{sl}_2)$:
\begin{align*}
x^i y^j z^k \qquad \qquad i, k \in \mathbb N, \qquad \qquad j \in \mathbb Z.
\end{align*}
\end{lemma}
\begin{proof} Similar to the proof of Lemma \ref{lem:WbasisQ}.
\end{proof}

\begin{proposition} \label{lem:uq} For  $\xi \in \mathbb F$, the elements $x,y,z$ of $U_q(\mathfrak{sl}_2)$ satisfy the
$\mathbb Z_3$-symmetric down-up relations with parameters 
\begin{align*}
\alpha = q^2 + \xi, \qquad \qquad \beta = -q^2 \xi, \qquad \qquad \gamma = (1-q^2)(1-\xi).
\end{align*}
\end{proposition}
\begin{proof} We have
\begin{align*}
&y^2 x - (q^2 +\xi) yxy+q^2 \xi xy^2 +(q^2-1)(1-\xi)y \\
&\quad =(q^2-1)\xi \biggl( \frac{qxy-q^{-1}yx}{q-q^{-1}}-1\biggr)y  - (q^2-1) y  \biggl( \frac{qxy-q^{-1}yx}{q-q^{-1}}-1\biggr) \\
&\quad = 0.
\end{align*}
We also have
\begin{align*}
&y x^2 - (q^2 +\xi) xyx+q^2 \xi x^2y +(q^2-1)(1-\xi)x \\
&\quad =(q^2-1)\xi x\biggl( \frac{qxy-q^{-1}yx}{q-q^{-1}}-1\biggr)  - (q^2-1)   \biggl( \frac{qxy-q^{-1}yx}{q-q^{-1}}-1\biggr)x \\
&\quad = 0.
\end{align*}
The result follows by these comments and $\mathbb Z_3$-symmetry.
\end{proof}

\begin{theorem} Pick $\xi \in \mathbb F$ and define
\begin{align*}
\alpha = q^2 + \xi, \qquad \qquad \beta = -q^2 \xi, \qquad \qquad \gamma = (1-q^2)(1-\xi).
\end{align*}
 There exists an algebra homomorphism $ \mathbb A(\alpha, \beta, \gamma) \to U_q(\mathfrak{sl}_2)$ that sends
\begin{align*}
A \mapsto x, \qquad \qquad B\mapsto  y, \qquad \qquad C \mapsto z.
\end{align*}
\end{theorem}
\begin{proof} By Proposition \ref{lem:uq}.
\end{proof}

\noindent We just described a relationship between $\mathbb A$ and $U_q(\mathfrak{sl}_2)$. There is  another
relationship between $\mathbb A$ and  $U_q(\mathfrak{sl}_2)$, which we discuss next.
\medskip

\noindent By Definition \ref{def:uq}, the following equations hold in  $U_q(\mathfrak{sl}_2)$:
\begin{align*}
q(1-xy)&=q^{-1}(1-yx),\\
q(1-yz)&=q^{-1}(1-zy), \\
q(1-zx)&=q^{-1}(1-xz).
\end{align*}

\noindent Following \cite[Definition~3.1]{uawe}, we define 
\begin{align*}
&\nu_x = q(1-yz) =q^{-1}(1-zy),
\\
&\nu_y = q(1-zx)=q^{-1}(1-xz),
\\
&\nu_z = q(1-xy)=q^{-1}(1-yx).
\end{align*}


\begin{lemma}
\label{lem:qcom1} \label{lem:step1} {\rm (See \cite[Lemma~3.5]{uawe}.)}
The following relations hold in $U_q(\mathfrak{sl}_2)$:
\begin{align*}
&x \nu_y = q^2 \nu_y x, \qquad \qquad
x \nu_z = q^{-2} \nu_z x,
\\
&y \nu_z = q^2 \nu_z y, \qquad \qquad
y \nu_x = q^{-2} \nu_x y,
\\
&z \nu_x = q^2 \nu_x z, \qquad \qquad
z \nu_y = q^{-2} \nu_y z.
\end{align*}
\end{lemma}

\begin{lemma} 
\label{lem:comnxny1} \label{lem:step2} {\rm (See \cite[Lemma~3.10]{uawe}.)}
The following relations hold in $U_q(\mathfrak{sl}_2)$:
\begin{align*}
\frac{q\nu_x\nu_y - q^{-1} \nu_y \nu_x}{q-q^{-1}} &= 1-z^2,
\\
\frac{q\nu_y\nu_z - q^{-1} \nu_z \nu_y}{q-q^{-1}} &= 1-x^2,
\\
\frac{q\nu_z\nu_x - q^{-1} \nu_x \nu_z}{q-q^{-1}} &= 1-y^2.
\end{align*}
\end{lemma}

\noindent The following result appears in \cite{agl}; we give a short proof for the sake of completeness.

\begin{lemma} \label{lem:step3} {\rm (See \cite[Lemma~3.9]{agl}.)} 
The following equations hold in $U_q(\mathfrak{sl}_2)$:
\begin{align*}
q^3 \nu_x^2 \nu_y - (q+q^{-1}) \nu_x \nu_y \nu_x + q^{-3} \nu_y \nu_x^2 &= (q^2 - q^{-2})(q-q^{-1}) \nu_x,\\
q^3 \nu_y^2 \nu_z - (q+q^{-1}) \nu_y \nu_z \nu_y + q^{-3} \nu_z \nu_y^2 &= (q^2 - q^{-2})(q-q^{-1}) \nu_y,\\
q^3 \nu_z^2 \nu_x - (q+q^{-1}) \nu_z \nu_x \nu_z + q^{-3} \nu_x \nu_z^2 &= (q^2 - q^{-2})(q-q^{-1}) \nu_z
\end{align*}
and also
\begin{align*}
q^{3} \nu_x \nu_y^2 - (q+q^{-1}) \nu_y \nu_x \nu_y + q^{-3} \nu_y^2 \nu_x&= (q^2 - q^{-2})(q-q^{-1}) \nu_y,\\
q^{3} \nu_y \nu_z^2 - (q+q^{-1}) \nu_z \nu_y \nu_z + q^{-3} \nu_z^2 \nu_y&= (q^2 - q^{-2})(q-q^{-1}) \nu_z,\\
q^{3} \nu_z \nu_x^2 - (q+q^{-1}) \nu_x \nu_z \nu_x + q^{-3} \nu_x^2 \nu_z&= (q^2 - q^{-2})(q-q^{-1}) \nu_x.\\
\end{align*}
\end{lemma}
\begin{proof} To verify the first equation, use Lemmas \ref{lem:step1}, \ref{lem:step2} to obtain
\begin{align*}
&\frac{q^3 \nu_x^2 \nu_y - (q+q^{-1}) \nu_x \nu_y \nu_x + q^{-3} \nu_y \nu_x^2 -(q^2 - q^{-2})(q-q^{-1}) \nu_x }{q-q^{-1}}\\
& \quad =  q^2 \nu_x \biggl( \frac{q \nu_x \nu_y - q^{-1} \nu_y \nu_x}{q-q^{-1}}-1+z^2     \biggr) 
 - q^{-2} \biggl( \frac{q \nu_x \nu_y - q^{-1} \nu_y \nu_x}{q-q^{-1}}-1+z^2   \biggr) \nu_x \\
& \qquad \qquad \qquad \qquad \qquad \qquad \qquad \qquad \qquad \qquad + q^{-2} z^2 \nu_x - q^2 \nu_x z^2 
\\
&\quad =0.
\end{align*}
The other equations are similarly verified.
\end{proof}

\begin{corollary}\label{cor:step4}  The elements $\nu_x, \nu_y, \nu_z$ of $U_q(\mathfrak{sl}_2)$ satisfy the $\mathbb Z_3$-symmetric down-up relations 
with parameters
\begin{align*}
\alpha = q^3(q+q^{-1}), \qquad \quad \beta = -q^6, \qquad \quad \gamma = q^{3}(q-q^{-1})(q^2-q^{-2}).
\end{align*}
\end{corollary}
\begin{proof} By Lemma \ref{lem:step3}.
\end{proof}

\begin{theorem}  We refer to the algebra $\mathbb A=\mathbb A(\alpha, \beta, \gamma)$ with
\begin{align*}
\alpha = q^3(q+q^{-1}), \qquad \quad \beta = -q^6, \qquad \quad \gamma = q^{3}(q-q^{-1})(q^2-q^{-2}).
\end{align*}
There exists an algebra homomorphism $ \mathbb A \to U_q(\mathfrak{sl}_2)$ that sends
\begin{align*}
A \mapsto \nu_x, \qquad \qquad B \mapsto \nu_y, \qquad \qquad C \mapsto \nu_z.
\end{align*}
\end{theorem} 
\begin{proof} By Corollary \ref{cor:step4}.
\end{proof}

\begin{note}\label{note:AGL} \rm The elements $\nu_x,\nu_y, \nu_z$ generate a subalgebra of $U_q(\mathfrak{sl}_2)$ called the
positive even part \cite[Proposition~5.4]{bocktingTer}.  This subalgebra has an attractive presentation by generators and relations \cite[Section~5]{agl}.
\end{note}

\section{How $\mathbb A$ is related to $U_q(A^{(1)}_2)$}

\noindent Recall the algebra $\mathbb A=\mathbb A(\alpha, \beta, \gamma)$ from Definition \ref{def:bbA}, and the 
 Kac-Moody Lie algebra $A^{(1)}_2$ from Definition \ref{def:km}. In this section, 
 we describe how $\mathbb A$ is related to the quantized enveloping algebra $U_q(A^{(1)}_2)$. 
\medskip

%
\noindent Throughout this section, we fix a nonzero $q \in \mathbb F$ such that $q^2 \not=1$.
\medskip

\noindent Recall the Cartan matrix $\mathcal C$  from \eqref{eq:Cartan}.

\begin{definition}\label{def:kmQ} \rm (See \cite[p.~281]{cp3}.) Define the  algebra $U_q(A^{(1)}_2)$ by generators
\begin{align*}
E_i, \quad F_i, \quad K_i, \quad K^{-1}_i \qquad \qquad (1 \leq i \leq 3)
\end{align*}
and the following relations. For $1 \leq i,j\leq 3$,
\begin{align} \label{eq:km1Q}
& K_i K^{-1}_i = K^{-1}_i K_i = 1, \qquad \qquad K_i K_j = K_j K_i,\\
& K_i E_j K^{-1}_i = q^{\mathcal C_{i,j}} E_j, \qquad \qquad  K_i F_j K^{-1}_i = q^{-\mathcal C_{i,j}} F_j, 
 \label{eq:km2Q} \\
 & E_i F_j - F_j E_i = \delta_{i,j} \frac{K_i - K^{-1}_i }{q-q^{-1}},
 \label{eq:km3Q}
 \\
& E^2_i E_j - (q+q^{-1}) E_i E_j E_i + E_j E^2_i = 0  \qquad \qquad (i \not=j), \label{eq:km4Q}\\
& F^2_i F_j - (q+q^{-1}) F_i F_j F_i + F_j F^2_i = 0 \;\,\qquad \qquad (i \not=j). \label{eq:km5Q}
\end{align}
We call $U_q(A^{(1)}_2)$  the {\it quantized enveloping algebra of $A^{(1)}_2$}.
\end{definition}

\noindent The following basic  facts about $U_q(A^{(1)}_2)$ can be found in \cite{cp3}. The element $K=K_1K_2K_3$ is central in $U_q(A^{(1)}_2)$. 
The elements 
\begin{align*}
 K^h_1 K^i_2 K^j_3 \qquad \qquad h,i,j \in \mathbb Z
\end{align*}
 form a basis for a commutative subalgebra $N_0$ of $U_q(A^{(1)}_2)$.
 Let $N_+$ (resp. $N_-$) denote the subalgebra of $U_q(A^{(1)}_2)$ generated by $\lbrace E_i \rbrace_{i=1}^3$ (resp. $\lbrace F_i \rbrace_{i=1}^3$).
The algebra $N_+$ (resp. $N_-$) is called the positive part (resp. negative part) of $U_q(A^{(1)}_2)$.
The multiplication map  
\begin{align*}
N_+ \otimes N_0 \otimes N_- \quad & \to  \quad
U_q(A^{(1)}_2) \\
     x \otimes y \otimes z\qquad &\mapsto \qquad xyz
\end{align*}
is an isomorphism of vector spaces.
The algebra $N_+$ has a presentation by generators $\lbrace E_i \rbrace_{i=1}^3$
and relations
\begin{align}
& E^2_i E_j - (q+q^{-1}) E_i E_j E_i + E_j E^2_i = 0  \qquad \quad (1 \leq i,j\leq 3,\;i \not=j).
 \label{eq:SerreEQ}
\end{align}
The algebra $N_-$ has a presentation by generators $\lbrace F_i \rbrace_{i=1}^3$
and relations
\begin{align*}
& F^2_i F_j - (q+q^{-1}) F_i F_j F_i + F_j F^2_i = 0 \;\,\qquad \quad (1 \leq i,j\leq 3, \;i \not=j).
\end{align*}
\noindent We now give some results about $N_+$; similar results hold for $N_-$.
\begin{lemma}\label{lem:sumEFQ} There exists an algebra isomorphsm $\mathbb A(q+q^{-1}, -1,0) \to N_+$ that sends
\begin{align*}
A \mapsto E_1, \qquad \qquad B\mapsto E_2, \qquad \qquad C\mapsto E_3.
\end{align*}
\end{lemma}
\begin{proof} The relations \eqref{eq:SerreEQ} are the $\mathbb Z_3$-symmetric down-up relations with parameters
\begin{align*}
\alpha = q+ q^{-1}, \qquad \quad \beta = -1, \qquad \qquad \gamma=0.
\end{align*}
\end{proof}

\begin{corollary} \label{lem:kmInjQ} There exists an algebra homomorphism $ \mathbb A(q+q^{-1}, -1,0) \to U_q(A^{(1)}_2)$ that sends
\begin{align*}
A \mapsto E_1, \qquad \qquad B\mapsto E_2, \qquad \qquad C\mapsto E_3.
\end{align*}
This homomorphism is injective.
\end{corollary}
\begin{proof}  By Lemma \ref{lem:sumEFQ}.
\end{proof}

\noindent Recall the element $K=K_1K_2K_3$.
\begin{definition}\label{def:ABCsQ} \rm  Let $\xi \in \mathbb F$. We  define some elements $A, B, C$ in $U_q(A^{(1)}_2)$ as follows: 
\begin{align*}
A = \bigl(E_1   +\xi  F_2 K^{-1}\bigr) K_3, \qquad B= \bigl(E_2 +\xi F_3 K^{-1}\bigr) K_1, \qquad  C= \bigl(E_3 +\xi F_1K^{-1}\bigr)K_2.
\end{align*}
\end{definition}
\begin{proposition} \label{lem:factA2Q} The elements $A, B, C$ from Definition \ref{def:ABCsQ} satisfy the $\mathbb Z_3$-symmetic down-up relations with parameters
\begin{align*} 
\alpha = q^3(q+q^{-1}), \qquad \quad \beta = - q^6, \qquad \quad \gamma = - \xi q^3 (q+q^{-1}).
\end{align*}
\end{proposition}
\begin{proof} This is routinely checked using the relations  \eqref{eq:km1Q}--\eqref{eq:km5Q}.
\end{proof}

\begin{theorem}\label{prop:efAQ} Let $\xi \in \mathbb F$ and define 
\begin{align*} 
\alpha = q^3(q+q^{-1}), \qquad \quad \beta = - q^6, \qquad \quad \gamma = - \xi q^3 (q+q^{-1}).
\end{align*}
With reference to Definition \ref{def:ABCsQ}, there exists an algebra homomorphism $ \mathbb A(\alpha, \beta,\gamma) \to U_q(A^{(1)}_2)$ that sends
\begin{align*}
A \mapsto A, \qquad \qquad B \mapsto B, \qquad \qquad C \mapsto C.
\end{align*}
\end{theorem}
\begin{proof} By Proposition \ref{lem:factA2Q}.
\end{proof}
\noindent We mention a variation on Theorem \ref{prop:efAQ}.

\begin{definition}\label{def:ABCsQ2} \rm  Let $\xi \in \mathbb F$. We  define some elements $A, B, C$ in $U_q(A^{(1)}_2)$ as follows: 
\begin{align*}
A = \bigl(E_1   +\xi  F_2 K \bigr) K^{-1}_3, \qquad B= \bigl(E_2 +\xi F_3 K \bigr) K^{-1}_1, \qquad  C= \bigl(E_3 +\xi F_1K \bigr)K^{-1}_2.
\end{align*}
\end{definition}

\begin{proposition} \label{lem:factA2Q2} The elements $A, B, C$ from Definition \ref{def:ABCsQ2} satisfy the $\mathbb Z_3$-symmetic down-up relations with parameters
\begin{align*}
\alpha = q^{-3}(q+q^{-1}), \qquad \quad \beta = - q^{-6}, \qquad \quad \gamma = - \xi q^{-3} (q+q^{-1}).
\end{align*}
\end{proposition}
\begin{proof} This is routinely checked using the relations  \eqref{eq:km1Q}--\eqref{eq:km5Q}.
\end{proof}

\begin{theorem}\label{prop:efAQ2} Let $\xi \in \mathbb F$ and define
\begin{align*} 
\alpha = q^{-3}(q+q^{-1}), \qquad \quad \beta = - q^{-6}, \qquad \quad \gamma = - \xi q^{-3} (q+q^{-1}).
\end{align*}
With reference to Definition \ref{def:ABCsQ2}, there exists an algebra homomorphism $ \mathbb A(\alpha, \beta, \gamma) \to U_q(A^{(1)}_2)$ that sends
\begin{align*}
A \mapsto A, \qquad \qquad B \mapsto B, \qquad \qquad C \mapsto C.
\end{align*}
\end{theorem}
\begin{proof} By Proposition \ref{lem:factA2Q2}.
\end{proof}

\section{Some observations} 
\noindent Recall the algebra $\mathbb A(\alpha, \beta, \gamma)$ from Definition
 \ref{def:bbA}.
In Proposition \ref{thm:three} we showed that $\mathbb A(\alpha, \beta, \gamma)$ is finite-dimensional and commutative, provided that $\alpha, \beta$ are zero and $\gamma$ is nonzero.
In this section, we consider how the conclusion changes if the provision is removed.

\begin{proposition} The algebra $\mathbb A(\alpha, \beta, \gamma)$ is infinite-dimensional and noncommutative, unless $\alpha=0$ and $\beta =0$ and $\gamma\not=0$.
\end{proposition}
\begin{proof} Abbreviate $\mathbb A = \mathbb A(\alpha, \beta, \gamma)$. We break the argument into four cases.
\medskip

\noindent Case $\gamma=0$:  By Corollaries   \ref{cor:NC} and \ref{lem:gamZ}. 
\medskip

\noindent Case $\gamma\not=0$ and $\alpha \not=0$: We invoke Lemma \ref{lem:ANZ}.  Let $\natural $ denote the homomorphism from that lemma.
 By matrix multiplication, the map $\natural$ sends
\begin{align*}
ABCA \mapsto -\alpha^{-3} \gamma^3 t^3 \natural (A).
\end{align*}
By this and induction,  the map $\natural$ sends
\begin{align*}
(ABC)^n A \mapsto (-1)^n \alpha^{-3n} \gamma^{3n} t^{3n} \natural (A), \qquad \qquad n \in \mathbb N.
\end{align*}
These images are linearly independent, so the  elements $\lbrace (ABC)^n A\rbrace_{n \in \mathbb N}$ of $\mathbb A$ are linearly independent. Therefore, $\mathbb A$ is infinite-dimensional.
One checks that $\natural (AB) \not=\natural (BA)$, so $AB\not=BA$. Therefore,
the algebra $\mathbb A$ is noncommutative. 
\medskip

\noindent Case $\gamma\not=0$ and $\alpha=0$ and $\beta=1$: 
We invoke Lemma \ref{lem:Wbasis} and Theorem \ref{prop:XiVal}.
Define $\theta = -\gamma/2$.  From Theorem \ref{prop:XiVal} with $\xi=-1$, we obtain a surjective algebra homomorphism $\mathbb A\to \mathbb W(\theta)$.
The algebra $\mathbb W(\theta)$ is infinite-dimensional by Lemma \ref{lem:Wbasis}, and  noncommutative since $\theta \not=0$.
By these comments, the algebra $\mathbb A$ is infinite-dimensional
and noncommutative. 
\smallskip
       
\noindent Case $\gamma\not=0$ and $\alpha=0$ and $\beta\not=1$ and $\beta \not=0$: We invoke Lemma \ref{lem:WbasisQ} and
Theorem  \ref{prop:QWmap}.  Define $ q \in \mathbb F$ such that $q^2 = \beta^{-1}$.  Note that $q\not=0$ and $ q \not=\pm 1$. Define $\theta = -q^2 \gamma /(q+1)$.
From Theorem \ref{prop:QWmap} with $\xi = -q^{-2}$ we obtain a surjective algebra homomorphism $\mathbb A \to \mathbb W_q(\theta)$.
The algebra $\mathbb W_q(\theta)$ is infinite-dimensional by Lemma \ref{lem:WbasisQ}, and noncommutative by construction. By these comments,
the algebra $\mathbb A$ is infinite-dimensional and noncommutative.
\end{proof}

\section{Lowering-raising triples}

\noindent  In Section 1, we mentioned that the algebra $\mathbb A(\alpha, \beta, \gamma)$ was
motivated by the concept of a lowering-raising triple of linear transformations. We now explain this motivation.
\medskip

\noindent 
The concept of a lowering-raising (or LR) triple was introduced in \cite{lrt}. An LR triple is defined as follows.
Let $d \in \mathbb N$.
Let $V$ denote a vector space with dimension $d+1$. Let ${\rm End}(V)$ denote the algebra
consisting of the $\mathbb F$-linear maps from $V$ to $V$.
By a {\it decomposition of
$V$} we mean a sequence  $\lbrace V_i\rbrace_{i=0}^d$
of one-dimensional subspaces whose direct sum is $V$.
Let  $\lbrace V_i\rbrace_{i=0}^d$
denote a decomposition of $V$. A linear transformation
$A \in {\rm End}(V)$ is said to {\it lower
$\lbrace V_i\rbrace_{i=0}^d$} whenever
$AV_i = V_{i-1} $ for $1 \leq i \leq d$ and
$AV_0 = 0$.
The map $A$ is said to {\it raise
$\lbrace V_i\rbrace_{i=0}^d$} whenever
$AV_i = V_{i+1} $ for $0 \leq i \leq d-1$ and
$AV_d = 0$.
An ordered pair of elements $A,B$ in
${\rm End}(V)$ is called {\it lowering-raising} (or {\it LR}) whenever
there exists  a decomposition of $V$ that is lowered by
$A$ and raised by $B$.
A 3-tuple of elements $A,B,C$ in
${\rm End}(V)$ is called an {\it LR triple} whenever
any two of $A,B,C$ form an LR  pair on $V$. 
The LR triple $A,B,C$ is said to be {\it over $\mathbb F$} and
have {\it diameter $d$}. 
\medskip

\noindent  In \cite{lrt} we showed how to normalize an LR triple,
and we classified up to isomorphism the normalized LR triples. We now summarize this classification, assuming $d\geq 2$ to avoid trivialities.
In \cite[Sections~26--30]{lrt} we displayed nine families of
normalized 
  LR triples over
$\mathbb F$ that have diameter $d$, denoted
\begin{align*}
&
{\rm NBWeyl}^+_d(\mathbb F;j,q),
\qquad \quad
{\rm NBWeyl}^-_d(\mathbb F;j,q),
\qquad \quad
{\rm NBWeyl}^-_d(\mathbb F;t),
\\
&
{\rm NBG}_d(\mathbb F;q),
\qquad \qquad \quad
{\rm NBG}_d(\mathbb F;1),
\\
&
{\rm NBNG}_d(\mathbb F;t),
\\
&
{\rm B}_d(\mathbb F;t,\rho_0,\rho'_0,\rho''_0),
\qquad \quad 
{\rm B}_d(\mathbb F;1,\rho_0,\rho'_0,\rho''_0),
\qquad \quad 
{\rm B}_2(\mathbb F;\rho_0,\rho'_0,\rho''_0).
\end{align*}
We showed that each 
normalized LR triple over $\mathbb F$ with diameter $d$ is isomorphic to
exactly one of these examples. 
\medskip

\noindent We just displayed nine families of LR triples. Members of the last three families have a property called bipartite \cite[Section~16]{lrt}. Members of the first six families are not bipartite \cite[Section~26]{lrt}.
\medskip

\noindent Let $A,B,C$ denote a normalized LR triple over $\mathbb F$ that has diameter $d$.
According to \cite[Section~32]{lrt}, the elements $A,B,C$ satisfy the following relations.
\medskip

\noindent ${\rm NBWeyl}^{\pm}_d(\mathbb F;j,q)$:
\begin{align*}
qAB-q^{-1}BA=\vartheta_j I, \qquad \quad
qBC-q^{-1}CB=\vartheta_j I, \qquad \quad
qCA-q^{-1}AC=\vartheta_j I,
\end{align*}
where 
$\vartheta_j = q^{-2j-1}(1+ q^{2j+1})^2(q-q^{-1})^{-1}$.
\medskip

\noindent ${\rm NBWeyl}^-_d(\mathbb F;t)$:
\begin{align*}
AB-tBA =\frac{2t}{1-t}I, \qquad \quad
BC-tCB=\frac{2t}{1-t}I, \qquad \quad
CA-tAC =\frac{2t}{1-t}I.
\end{align*}
${\rm NBG}_d(\mathbb F;q)$: \quad
$A,B,C$ satisfy the $\mathbb Z_3$-symmetric down-up relations with parameters
\begin{align*}
\alpha = q^{-2}(q+1), \qquad \qquad 
\beta = -q^{-3}, \qquad \qquad 
\gamma= q^{-2}(q+1).
\end{align*}
\noindent ${\rm NBG}_d(\mathbb F;1)$: \quad $A,B,C$ satisfy the $\mathbb Z_3$-symmetric down-up relations with parameters
\begin{align*}
\alpha = 2, \qquad \qquad 
\beta = -1, \qquad \qquad 
\gamma= 2.
\end{align*}
\noindent ${\rm NBNG}_d(\mathbb F;t)$: \quad $A,B,C$ satisfy the $\mathbb Z_3$-symmetric down-up relations with parameters
\begin{align*}
\alpha = 0, \qquad \qquad 
\beta = t^{-1}, \qquad \qquad 
\gamma= t^{-1}-1.
\end{align*}
\noindent ${\rm B}_d(\mathbb F;t,\rho_0,\rho'_0,\rho''_0)$:
\begin{align*}
&A^3 B + A^2BA-tABA^2 -tBA^3 = (\rho_0 + t/\rho_0)A^2,
\\
&
B^3 C + B^2CB-tBCB^2 -tCB^3 = (\rho'_0 + t/\rho'_0)B^2,
\\
&
C^3 A + C^2AC-tCAC^2 -tAC^3 = (\rho''_0 + t/\rho''_0)C^2
\end{align*}
and also
\begin{align*}
&
A B^3 + BAB^2-tB^2AB -tB^3A = (\rho_0 + t/\rho_0)B^2,
\\
&
B C^3 + CBC^2-tC^2BC -tC^3B = (\rho'_0 + t/\rho'_0)C^2,
\\
&
C A^3 + ACA^2-tA^2CA -tA^3C = (\rho''_0 + t/\rho''_0)A^2.
\end{align*}
\noindent  ${\rm B}_d(\mathbb F;1,\rho_0,\rho'_0,\rho''_0)$:
\begin{align*}
&A^3 B + A^2BA-ABA^2 -BA^3 = (\rho_0 + 1/\rho_0)A^2,
\\
&
B^3 C + B^2CB-BCB^2 -CB^3 = (\rho'_0 + 1/\rho'_0)B^2,
\\
&
C^3 A + C^2AC-CAC^2 -AC^3 = (\rho''_0 + 1/\rho''_0)C^2
\end{align*}
and also
\begin{align*}
&
A B^3 + BAB^2-B^2AB -B^3A = (\rho_0 + 1/\rho_0)B^2,
\\
&
B C^3 + CBC^2-C^2BC -C^3B = (\rho'_0 + 1/\rho'_0)C^2,
\\
&
C A^3 + ACA^2-A^2CA -A^3C = (\rho''_0 + 1/\rho''_0)A^2.
\end{align*}
\noindent
${\rm B}_2(\mathbb F;\rho_0,\rho'_0,\rho''_0)$: \quad Same as ${\rm B}_d(\mathbb F;1,\rho_0,\rho'_0,\rho''_0)$.
\medskip

\noindent  In summary, we have the following result.
\begin{proposition}\label{thm:LRT} Let $A,B,C$ denote a normalized LR triple over $\mathbb F$ that has diameter $d\geq 2$.
Assume that $A,B,C$ is not bipartite. Then there exist $\alpha, \beta, \gamma \in \mathbb F$ such that
$A,B,C$ satisfy the $\mathbb Z_3$-symmetric down-up relations with parameters $\alpha, \beta, \gamma$.
\end{proposition}


\section{Directions for future research}

\noindent In this section, we list some problems and conjectures concerning the $\mathbb Z_3$-symmetric down-up algebra $\mathbb A=\mathbb A(\alpha, \beta, \gamma)$.

\begin{problem}\rm We refer to the algebra $\mathbb A_0$ in Lemma \ref{lem:A0}. Find a presentation of  $\mathbb A_0$ by generators and relations, using the generators listed in \eqref{eq:list}.
\end{problem}

\begin{problem}\rm Determine the values of $\alpha, \beta, \gamma$ such that the map  $\mathcal A(\alpha, \beta, \gamma)\to \mathbb A(\alpha, \beta, \gamma)$ from Lemma \ref{lem:natH}
is injective.
\end{problem}

\begin{problem}\rm Find the center of $\mathbb A$.
\end{problem}

\begin{problem}\rm Find a basis for the vector space $\mathbb A$.
\end{problem}

\begin{problem}\rm Describe the finite-dimensional irreducible $\mathbb  A$-modules on which $A,B,C$ are nilpotent.
\end{problem}
\begin{problem}\rm Describe the finite-dimensional irreducible $\mathbb  A$-modules on which $A,B,C$ are invertible.
\end{problem}

\begin{problem}\rm In Sections 17 and 18, we described how the algebra $\mathbb A$ is related to 
$U_q(\mathfrak{sl}_2)$ and $U_q(A^{(1)}_2)$. The results of \cite{benkWith, benkW2} suggest that $\mathbb A$ is similarly related to some two-parameter quantum groups $U_{r,s}(\mathfrak{g})$.
It would be interesting to explore this issue.
\end{problem}

\begin{problem}\rm We refer to Section 20. Let $A,B,C$ denote a normalized LR triple over $\mathbb F$ that has diameter $d\geq 2$.
Assume for the moment that $A,B,C$ is not bipartite. By Proposition \ref{thm:LRT} the maps $A,B,C$ satisfy some $\mathbb Z_3$-symmetric down-up relations. 
Next assume that $A,B,C$ is bipartite. By the display above Proposition \ref{thm:LRT}, the maps $A,B,C$
satisfy some  $\mathbb Z_3$-symmetric polynomial equations that have total degree $4$. Perhaps these equations
can be interpreted using a Lie superalgebra or its quantized enveloping algebra. It would be interesting to explore this issue. The book \cite{musson}
gives an introduction to Lie superalgebras.
\end{problem}

\begin{conjecture}\rm  The map in Proposition \ref{lem:LR} 
is injective.
\end{conjecture}

\begin{conjecture}\rm The maps in Proposition \ref{prop:natt}
and Theorem \ref{prop:2nat} are injective.
\end{conjecture}

\begin{problem}\rm 
The map in Corollary \ref{lem:kmInj2} is injective.
This result suggests that the map in 
Theorem \ref{cor:AandKM} is injective. It would be interesting to explore this issue.
\end{problem}

\begin{problem} \rm The map in Corollary \ref{lem:kmInjQ} is injective. This result suggests that the maps in Theorems \ref{prop:efAQ}, \ref{prop:efAQ2} are injective. It would be interesting to
explore this issue.
\end{problem}




\bigskip

\noindent Paul Terwilliger \hfil\break
\noindent Department of Mathematics \hfil\break
\noindent University of Wisconsin \hfil\break
\noindent 480 Lincoln Drive \hfil\break
\noindent Madison, WI 53706-1388 USA \hfil\break
\noindent Email: {\tt terwilli@math.wisc.edu }\hfil\break
\bigskip

 \bigskip
 
\section{Statements and Declarations}

\noindent {\bf Funding}: The authors declare that no funds, grants, or other support were received during the preparation of this manuscript.
\medskip

\noindent  {\bf Competing interests}:  The authors  have no relevant financial or non-financial interests to disclose.
\medskip

\noindent {\bf Data availability}: All data generated or analyzed during this study are included in this published article.
 \end{document}